\documentclass{paper}
\usepackage{a4wide}
\usepackage{cite}
\usepackage{amsmath}
\usepackage{amssymb}
\usepackage{amstext}
\usepackage{amsthm}
\usepackage{graphicx}
\usepackage{comment}
\usepackage{bbm}

\usepackage{tabularx}
\usepackage{tikz}
\usepackage{pgfplots}
\usepgfplotslibrary{external}
\usepackage{calc}

\usepackage{xurl}
\usepackage{hyperref}
\usepackage{subcaption}
\usepackage{mathtools}

\usepackage[ruled,vlined]{algorithm2e}

\newcommand{\R}{\mathbb{R}}

\newtheorem{proposition}{Proposition}[section]

\theoremstyle{definition}
\newtheorem{definition}{Definition}[section]
\newtheorem{remark}{Remark}[section]
\newtheorem{ex}{Example}[section]

\numberwithin{equation}{section}
\allowdisplaybreaks[4]

\title{A multi-class kinetic traffic flow model: discrete-velocity formulation and diffusively-corrected macroscopic limits}
\date{\today}
\author{C. Mezquita-Nieto\,\footnotemark[1] \footnotemark[2] \footnotemark[3] \and P. Goatin\,\footnotemark[2] \and A. Klar\,\footnotemark[1]}

\begin{document}

\maketitle
\renewcommand{\thefootnote}{\fnsymbol{footnote}}
\footnotetext[1]{RPTU Kaiserslautern, Department of Mathematics, 67663 Kaiserslautern, Germany (\{carmen.mezquitanieto, klar\}@rptu.de)}
\footnotetext[2]{Universit\'e C\^ote d'Azur, Inria, CNRS, LJAD, 2004 route des Lucioles - BP 93, 06902 Sophia Antipolis Cedex, France. E-mail: (\{carmen.mezquita-nieto, paola.goatin\}@inria.fr)}
\footnotetext[3]{Corresponding author.}

\begin{abstract}
	This paper introduces a multi-class extension of a discrete-velocity kinetic traffic flow model based on a non-local Prigogine-Herman framework. 
	We derive a hyperbolically scaled system of equations from a continuous kinetic formulation describing interactions between different vehicle classes through braking and relaxation terms.
	The model is then discretized with respect to the velocity variable for an arbitrary number of vehicle classes, and the structural properties of the resulting formulation are analyzed. 
	In particular, we prove hyperbolicity and total linear degeneracy. 
	Due to the  non-conservative structure of the model, we employ a path-conservative finite volume scheme for the numerical approximation of the system. 
	Finally, we derive the  corresponding diffusively-corrected macroscopic multi-class model, investigate its stability  and present numerical simulations on a single-lane road to illustrate the theoretical findings.
\end{abstract}

\textbf{Keywords:} Discrete-velocity model, multiclass kinetic traffic flow, relaxation system, totally linear degenerate hyperbolic equation, nonconservative products, path-conservative methods. 


\section{Introduction and motivation}
\label{sec_kin:introduction}

Traffic flow modeling has extensively employed both macroscopic and microscopic frameworks to capture the complex dynamics of vehicular networks. The foundational first-order Lighthill-Whitham-Richards (LWR) \cite{Light-Whit,Richards} model describes large-scale traffic phenomena through a scalar conservation law. To incorporate velocity dynamics and non-equilibrium behavior, second-order macroscopic models---such as the Payne-Whitham \cite{pay79, Whi74} and Aw-Rascle-Zhang \cite{AR, Zhang} models---introduce a momentum equation governed by a relaxation term toward an equilibrium velocity. While these macroscopic models efficiently describe large-scale traffic phenomena, they often lack the resolution to capture distinct driver interactions and velocity variances. Kinetic models bridge this gap by describing traffic dynamics through statistical distribution functions of vehicle positions and velocities. A common approach in this mesoscopic framework is the use of Bhatnagar-Gross-Krook (BGK)-type kinetic models \cite{BGK1954}, which describe the adaptation of driver velocities via relaxation operators toward a localized equilibrium distribution, see for example \cite{HPRV, KW97, KW00}. Within this context, the Prigogine-Herman kinetic model \cite{PH71} has served as a foundational framework, introducing relaxation and interaction terms to model how drivers adjust their speeds depending on their surroundings. The derivation of formal asymptotic limits and macroscopic model closures from such kinetic frameworks has been extensively investigated, see for instance \cite{Nel99, Sopasakis}.

In modern transportation networks, traffic is inherently heterogeneous, accounting for various vehicle types, such as passenger cars and heavy transport trucks, each characterized by distinct maximum speeds and acceleration profiles. While macroscopic traffic flow models have been widely treated in this setting, mainly through the generalization of the LWR model to the multi-species case \cite{Benzoni-Colombo, FanWork2015, WongWong2002}, extending single-class kinetic models to multi-class formulations remains essential. This approach is necessary for accurately capturing the intricate interactions between different vehicle classes, such as overtaking dynamics and class-specific braking behavior. Several recent works regarding multi-class kinetic models are present in \cite{BisiLoy24, Puppo16, Puppo17, Mendez2019}.

The mathematical formulation and numerical resolution of multi-class kinetic models present significant challenges. In the multi-class setting, interactions between different species often lead to systems of equations containing non-conservative products \cite{DalMaso1995DefinitionAW, LeFloch99NonconservProducts}, which complicates the definition of weak solutions across discontinuities. Consequently, standard conservative finite volume methods are insufficient, necessitating the use of path-conservative numerical schemes \cite{Pares06,Pares04Roe,Pares09,KolbeHerty24}. Furthermore, deriving the diffusively corrected macroscopic limit of these models yields systems governed by cross-diffusion. In this work, the structural stability of the resulting system is established following  \cite{DCMCLWR}, where stability criteria for a diffusively-corrected multi-class model are discussed to study the persistence of traffic wave structures.

To address these mathematical and numerical challenges, this paper proposes a novel multi-class discrete-velocity kinetic traffic flow model based on a modified non-local Prigogine-Herman framework, directly extending the single-class model presented in \cite{BK18, BK24discrete}. The original formulation possessed several desirable properties. In particular, it was a rich system \cite{peng, serre} that admitted a complete set of Riemann invariants and conservative variables. However, the extension of the model to multiple classes of vehicles results in the loss of several of these properties, making the mathematical formulation and numerical resolution significantly more complex.

The remainder of the article is organized as follows. Section \ref{sec_kin:kineticmodel} presents the continuous multi-class extension. In Section \ref{sec_kin:discretizationmodel}, we discretize the velocity space to derive a coupled system of equations for an arbitrary number of vehicle classes, and we rigorously analyze its structural properties, proving total linear degeneracy and identifying conservative variables. To handle the non-conservative terms arising from inter-class interactions, Section \ref{sec_kin:path_cons_methods} deploys a path-conservative finite volume scheme. Section \ref{sec_kin:stability_cont_model} derives the macroscopic limit of the model via a Chapman-Enskog expansion and provides a thorough stability analysis. Finally, Section \ref{sec_kin:numericalresults} presents the numerical algorithm and illustrative tests, and Section \ref{sec_kin:conclusions} discusses the results and future perspectives.

\section{A continuous kinetic model for multi-class  traffic flow}
\label{sec_kin:kineticmodel}

Our starting point is a multiclass modified version of the classical kinetic Prigogine-Herman equation for traffic flow, see \cite{KW97, KW981, KW00, Hel95B, Nel95, PF75, PH71}. We proceed similarly as in the single-class model \cite{BK18,BK24discrete}.

\subsection{Derivation of the continuous kinetic model}
\label{sec_kin:kineticmodel:derivation}

The kinetic model that will be treated is a multi-class extension of the traffic flow model presented in \cite{BK24discrete,BK24macro}. This approach stems from a microscopic model that accounts for braking interactions: a driver circulating along a stretch of road, situated at position $x$ with velocity $v$, reacts to the vehicle ahead of them at position $x+H$ with velocity $\hat{v}<v$, with $H>0$ denoting the minimum distance between vehicles.
The vehicle at position $x$ will adjust its velocity to match the one of its predecessor. Within this framework, $R=1/H$ is the maximum capacity of the road, that is, the maximum density of vehicles, and the interaction between the two vehicles, derived from the two-particle correlation function of vehicles in \cite{KW00}, is scaled by $(1-Hr)^{-1}$, where $r$ is the total density of all classes of vehicles. 

Let $t\in\R^+$, $x\in\R$ and $v\in[0,V]$. Take $N$ as the total number of vehicle classes considered. Then, the distribution functions $f^c=f^c(x,v,t)\ge 0$, with $f=(f^1,\ldots,f^N)$ follow the equation
\begin{align}
    \partial_t f^c + v\partial_x f^c = J_B^c(f) + J_A^c(f), \qquad c=1,\ldots,N.\label{eq_kin:kineticunscaled}
\end{align}
The functional $J^c_B$ denotes the interactions related to braking for a fixed class $c$ with respect to all other vehicle classes, including its own, while $J^c_A$ contains other interactions.
Regarding the definition of the braking term, following \cite{PH71,BKK,BK18,BK24discrete,BK24macro}, it is defined as 
\begin{align*}
    J_B^c(f)(x,v) = \frac{1}{1-Hr} \sum_{d=1}^N \Big(&\int_{\hat{v}>v} (\hat{v}-v) f^c(x,\hat{v}) f^d(x+H,v) d\hat{v}\\
    - &\int_{\hat{v}<v} (v-\hat{v}) f^c(x,v) f^d(x+H, \hat{v}) d\hat{v}\Big),
\end{align*}
with $r=r(x)=\sum_c \rho^c(x)=\int_0^{V}\sum_c f^c(x,v) dv\le R=\frac{1}{H}$, where $\rho^c$ denotes each individual class density, bounded by its respective maximal class density $R_c$, and $R$ denotes the maximum density over all maximum class densities, i.e., $R = \max \{ R_1, \ldots, R_N \}$.

Defining the distribution function $\tilde{f}(\tilde{x},\tilde{v},\tilde{t}) = \frac{V}{R} f(x,v,t) = H V f(x,v,t)$, with $\tilde{x}= \frac{x}{H}$, $\tilde{v} = \frac{v}{V}$ and $\tilde{t} =\frac{V}{H}t$, results in a modified \eqref{eq_kin:kineticunscaled} where, neglecting the $\sim$ notation,
\begin{align*}
    J_B^c(f)(x,v) = \frac{1}{1-r} \sum_{d=1}^N \Big(&\int_{\hat{v}>v} (\hat{v}-v) f^c(x,\hat{v}) f^d(x+1,v) d\hat{v}\\
    - &\int_{\hat{v}<v} (v-\hat{v}) f^c(x,v) f^d(x+1, \hat{v}) d\hat{v}\Big),
\end{align*}
with $t\in\R^+$, $x\in\R$, $v\in[0,1]$ and $0\le r\le 1$.

The class-specific braking term $J_B^c$ can now be divided into two: the first term is $J_B^{L,c}$, containing the local braking interactions (see \cite{PH71}), and $J_B^{NL,c}$ denoting the nonlocal ones. These can be seen as
\begin{align*}
    J_B^{L,c}(f)(x,v)=\frac{1}{1-r} \sum_{d=1}^N \Big(&\int_{\hat{v}>v} (\hat{v}-v) f^c(x,\hat{v}) f^d(x,v) d\hat{v}\\
    -&\int_{\hat{v}<v} (v-\hat{v}) f^c(x,v) f^d(x,\hat{v}) d\hat{v}\Big)
\end{align*}
and
\begin{equation*}
    \begin{aligned}
        J_B^{NL,c}(f)(x,v) &= J_B^c(f) - J_B^{L,c}(f)\\
        &
        \begin{aligned}
            \;= \frac{1}{1-r}\sum_{d=1}^N \Big(&\int_{\hat{v}>v} (\hat{v}-v) f^c(x,\hat{v})[f^d(x+1,v) - f^d(x,v)] d\hat{v}\\
            - &\int_{\hat{v}<v} (v-\hat{v}) f^c(x,v) [f^d(x+1,\hat{v}) - f^d(x,\hat{v})] d\hat{v}\Big).
        \end{aligned}
    \end{aligned}
\end{equation*}
Taking a hyperbolic space-time scaling, with $t\rightarrow t/\epsilon$ and $x \rightarrow x/\epsilon$, the kinetic equation is rewritten as
\begin{align}
    \partial_t f^c + v\partial_x f^c = \frac{1}{\epsilon} J_B^{NL,c}(f) + \frac{1}{\epsilon} \left(J_B^{L,c}(f) + J_A^c(f)\right), \label{eq_kin:kineticscaled_1}
\end{align}
with
\begin{align*}
    J_B^{NL,c}(f)(x,v)= \frac{1}{1-r}\sum_{d=1}^N \Big(&\int_{\hat{v}>v} (\hat{v}-v) f^c(x,\hat{v})[f^d(x+\epsilon,v) - f^d(x,v)] d\hat{v}\\
    - &\int_{\hat{v}<v} (v-\hat{v}) f^c(x,v) [f^d(x+\epsilon,\hat{v}) - f^d(x,\hat{v})] d\hat{v}\Big).
\end{align*}

\begin{remark}
    \label{rem_kin:kineticmodel:derivation:EnskogBoltzmann_scaling}
	In the fluid dynamic context, this corresponds to the classical hydrodynamic Enskog-Boltzmann scaling (see \cite{lacho}).
\end{remark}

Taking a Taylor approximation, it is possible to approximate the non-local term $J_B^{NL,c}(f)$ up to order $\mathcal{O}(\epsilon^2)$ by the term $\epsilon J^c(f,\partial_x f)$, with
\begin{align*}
    J^c(f,\partial_x f)(x,v) = \frac{1}{1-r} \sum_{d=1}^N \Big(&\int_{\hat{v}>v} (\hat{v}-v) f^c(x,\hat{v}) \partial_x f^d(x,v) d\hat{v}\\
    - &\int_{\hat{v}<v} (v-\hat{v}) f^c(x,v) \partial_x f^d(x,\hat{v}) d\hat{v}\Big),
\end{align*}
where $f^d(x+\epsilon,v) = f^d(x,v) + \epsilon \partial_x f^d(x,v) + \mathcal{O}(\epsilon^2)$, obtaining
\begin{align}
    \partial_t f^c + v\partial_x f^c - J^c(f,\partial_x f) = \frac{1}{\epsilon} \left(J_B^{L,c}(f) + J_A^c(f)\right). \label{eq_kin:kineticscaled_2}
\end{align}
In this article, following \cite{BK24discrete,BK24macro}, we will choose an approximation of the full-local term, corresponding to the right-hand side of the equation, by a relaxation term,
\begin{align*}
    J_B^{L,c}(f) + J_A^c(f) \approx J^c_R(f) = \frac{1}{T^c} \left(f^{e,c}(\rho,v) - f^c\right),
\end{align*}
where $f^{e,c}(\rho,v)\ge 0$ is an equilibrium function, with $\mathbf{\rho}=(\rho^1,\ldots,\rho^N)$, $\rho^c = \rho^c(x,t)$, for $0\le\rho^c\le 1$, $0\le v\le 1$ such that
\begin{equation*}
    \int f^{e,c}(\rho,v) dv = \rho^c = \int f^c(v) dv.
\end{equation*}
Altogether, we obtain the kinetic problem
\begin{align}
    \partial_t f^c + v\partial_x f^c - J^c(f, \partial_x f) = \frac{1}{\epsilon}J_R^c(f), \qquad c=1,\ldots,N, \label{eq_kin:kinbasic}
\end{align}
which will be the starting point of our investigations. Let the first and second moments of $f^{e,c}$ be denoted by $F^c$ and $E^c$, respectively, where
\begin{equation*}
    F^c(\rho) = \rho^c v_c(\rho) = \int vf^{e,c}(\rho,v)dv, \qquad E^c(\rho) = \int v^2 f^{e,c}(\rho,v)dv,
\end{equation*}
with $F^c(\rho)$ as the class-specific fundamental diagram. For this $F^c$, the $c$-equilibrium function $f^{e,c}$
must be chosen such that the class-velocity function $v_c(\mathbf{\rho}) = F^c(\mathbf{\rho})/\rho^c$ is monotone decreasing, with $v_c(0) = V_c\le 1$ and $v_c(1)=0$. Therefore, $F^c$ is a positive function, with $F^c(0) = F^c(1) = 0$ and $F^c(\rho)\le\rho^c$. 

\begin{ex}
    \label{ex_kin:cont_equil_fct}
    An example of a possible equilibrium function $f^{e,c}$ can be the one presented in \cite[Ex. 1]{BK18} applied to the multi-class case,
    \begin{equation}
        f^{e,c}(\rho,v) = \alpha_0^c(\rho) \delta_0(v) + (\rho - \alpha_0^c(\rho) - \alpha_1^c(\rho)) + \alpha_1^c(\rho) \delta_1(v), \label{eq_kin:cont_equil_fct}
    \end{equation}
    where 
    \begin{equation*}
        \alpha_0^c(\rho) = \rho^c - F^c - 3(F^c-E^c), \qquad \alpha_1^c(\rho) = F^c - 3(F^c-E^c).
    \end{equation*}
\end{ex}
For a simpler analytical discussion of the multi-class model, we use a change of variables.
We define 
\begin{equation}
    \bar{f}^c = f^1+\cdots+f^c, \qquad c=1,\ldots,N, \label{eq_kin:fbar_def_initial}
\end{equation}
i.e.,
\begin{equation}
    f^1 = \bar{f}^1 \qquad \text{and} \qquad f^c = \bar{f}^c - \bar{f}^{c-1}, \qquad c=2,\ldots,N. \label{eq_kin:fbar_def}
\end{equation}
Similarly, we will denote $\int \bar{f}^{e,c}dv = \bar{\rho}^c = \int \bar{f}^c dv$, where $\bar{f}^{e,c}$ denotes the equilibrium function associated to $\bar{f}^c$. In particular, as $\bar{f}^N=f^1+\cdots+f^N$, then $\int \bar{f}^Ndv = \bar{\rho}^N = r$. The terms $J^c(f,\partial_x f)$, for $c=1,\ldots,N$, would now be
\begin{equation*}
    J^1(f,\partial_x f) = J(\bar{f}^1,\partial_x \bar{f}^N) \qquad \text{and} \qquad
    J^c(f,\partial_x f) = J(\bar{f}^c-\bar{f}^{c-1}, \partial_x \bar{f}^N), \qquad c=2,\ldots,N,
\end{equation*}
where
\begin{align*}
    J(f,g)(x,v) = \frac{1}{1-r} \Big(&\int_{\hat{v}>v} (\hat{v}-v) f(x,\hat{v}) g(x,v) d\hat{v} - \int_{\hat{v}<v} (v-\hat{v}) f(x,v) g(x,\hat{v}) d\hat{v}\Big).
\end{align*}
This means that, for $c=1,\ldots,N$, $\sum_{k=1}^c J^k(f,\partial_x f) = J(\bar{f}^c, \partial_x \bar{f}^N)$. Therefore, 
 defining $\bar{J}^c_R(\bar{f}) = \sum_{k=1}^c J_R^k(f)$, the kinetic problem \eqref{eq_kin:kinbasic} is rewritten in terms of  the  variables \eqref{eq_kin:fbar_def} as
\begin{align}
    \partial_t \bar{f}^c + v\partial_x \bar{f}^c - J(\bar{f}^c, \partial_x \bar{f}^N) = \dfrac{1}{\epsilon} \bar{J}^c_R(\bar{f}), \qquad c=1,\ldots,N, \label{eq_kin:kinbasic_fbar}
\end{align}
with
\begin{align*}
    \bar{J}^c_R(\bar{f}) &= \dfrac{\bar{f}^{e,1} - \bar{f}^1}{T^1} + \sum_{k=2}^c \dfrac{(\bar{f}^{e,k}-\bar{f}^{e,k-1}) - (\bar{f}^k - \bar{f}^{k-1})}{T^k}\\
    &= \sum_{k=1}^{c-1} \left(\dfrac{1}{T^k}-\dfrac{1}{T^{k+1}}\right) (\bar{f}^{e,k} - \bar{f}^k) + \dfrac{\bar{f}^{e,c} - \bar{f}^c}{T^c}.
\end{align*}

\section{Discretization of the model}
\label{sec_kin:discretizationmodel}

Let us now consider discretized class densities $\rho^c$ such that
\begin{equation*}
    \sum_{\ell=0}^M f_{\ell}^c = \rho^c = \sum_{\ell=0}^M f_{\ell}^{e,c}(\rho), \qquad r = \sum_{c=1}^N \rho^c.
\end{equation*}
We can define a discrete velocity model for each class $c=1,\ldots,N$ with the distribution functions $f^c_i=f_i^c(x,t)\in[0,1]$, $i=0,\ldots,M$, and associated velocities $0=v_0<v_1<\cdots<v_{M-1}<v_M=1$. The discrete momentum is then $q^c = \sum_{i=0}^M v_i f_i^c$. Together with the equilibrium flux $F^c(\mathbf{\rho}) = \sum_{i=0}^M v_i f_i^{e,c}(\rho)$, the discrete velocity model is given by the system of kinetic equations
\begin{equation}
    \partial_t f_i^c + v_i \partial_x f_i^c - J^c_i(f,\partial_x f) = \frac{1}{\epsilon} J_R^{c,i}, \quad i=0,\ldots,M, \quad c=1,\ldots,N, \label{eq_kin:discretized_model_1}
\end{equation}
where $J^c_i(f,\partial_x f)$, $i=0,\ldots,M$, is associated to the discretization of $J^c(f,\partial_x f)$, and
\begin{align*}
    J_R^{c,i} = \frac{1}{T^c}(f_i^{e,c}(\mathbf{\rho}) - f_i^c).
\end{align*}
 $e^c$ and $E^c(\mathbf{\rho})$ are discretized as $e^c = \sum_{i=0}^M v_i^2 f_i^c$ and $E^c(\mathbf{\rho}) = \sum_{i=0}^M v_i^2 f_i^{e,c}(\rho)$, respectively.
\begin{ex}
    \label{ex_kin:discr_equil_fct}
    Following \eqref{eq_kin:cont_equil_fct}, let $v_i=i/M$ for $i=0,\ldots,M$. Then, the discrete equilibrium functions can be written, for each class $c=1,\ldots,N$, as
    \begin{equation}
        \begin{array}{rl}
            i=0: & f_0^{e,c}(\rho) = \rho^c - \displaystyle\sum_{i=1}^M f_i^{e,c}(\rho) = \rho^c - F^c - \dfrac{3M}{M+1}(F^c-E^c), \\[8pt]
            i=1,\ldots,M-1: & f_i^{e,c}(\rho) = \dfrac{6}{M-1} \dfrac{M}{M+1} (F^c-E^c), \\[8pt]
            i=M: & f_M^{e,c}(\rho) = F^c - \displaystyle\sum_{i=1}^{M-1} v_i f_i^{e,c}(\rho) = F^c - \dfrac{3M}{M+1} (F^c-E^c),
        \end{array}
        \label{eq_kin:discr_equil_fct}
    \end{equation}
    which corresponds to the discrete version of \eqref{eq_kin:cont_equil_fct}. In particular, for the 2-velocity case ($M=1$),
    \begin{equation*}
        f_0^{e,c}(\rho) = \rho^c - F^c, \qquad f_1^{e,c}(\rho) = F^c.
    \end{equation*}
\end{ex}
Analogously, we can define $\bar{\rho}^c$, $\bar{q}^c$, $\bar{F}^c(\bar{\rho})$, $\bar{e}^c$ and $\bar{E}^c(\bar{\rho})$. The discrete velocity model for the variables \eqref{eq_kin:fbar_def} would be
\begin{equation}
    \partial_t \bar{f}^c_i + v_i \partial_x \bar{f}_i^c - J_i(\bar{f}^c, \partial_x \bar{f}^N) = \dfrac{1}{\epsilon} \bar{J}_R^{c,i}, \quad i=0,\ldots,M, \quad c = 1,\ldots,N, \label{eq_kin:discretized_model_1_fbar}
\end{equation}
where $\bar{J}_R^{c,i}=\sum_{k=1}^c J_R^{k,i}$, with functions $\bar{f}^c_i=\bar{f}_i^c(x,t)\in[0,1]$. 

We can discretize the term $J(\bar{f}^c,\partial_x \bar{f}^N)$ in \eqref{eq_kin:kinbasic}, taking $I(\bar{f}^c,\partial_x \bar{f}^N)=(1-r)J(\bar{f}^c,\partial_x \bar{f}^N)$, as
\begin{equation}
    \begin{aligned}
        I_0(\bar{f}^c,\partial_x \bar{f}^N) &= \sum_{\ell=1}^M v_{\ell} \bar{f}_{\ell}^c \partial_x \bar{f}_0^N,\\
        I_i(\bar{f}^c,\partial_x \bar{f}^N) &= \sum_{\ell=i+1}^M (v_{\ell}-v_i) \bar{f}^c_{\ell} \partial_x \bar{f}_i^N - \sum_{\ell=0}^{i-1} (v_i-v_{\ell}) \bar{f}_i^c \partial_x \bar{f}_{\ell}^N, \qquad i=1,\ldots,M-1,\\
        I_M(\bar{f}^c,\partial_x \bar{f}^N) &= -\sum_{\ell=0}^{M-1} (1-v_{\ell}) \bar{f}_M^c \partial_x \bar{f}_{\ell}^N.
    \end{aligned}
    \label{eq_kin:interaction_term}
\end{equation}
Note that, for each equation $i=0,\ldots,M$ for a fixed class $c$, the partial derivatives with respect to $x$ depend only on distribution functions associated to class $N$. Defining $U$ and $\bar{J}_R(U)$ as
\begin{equation*}
    U = (U^1,\ldots,U^N)^T, \qquad U^c=(\bar{f}_0^c,\ldots,\bar{f}_M^c)^T
\end{equation*}
and
\begin{equation*}
    \bar{J}_R(U)=(\bar{J}_R^1,\ldots,\bar{J}_R^N)^T, \qquad \bar{J}_R^c=(\bar{J}_R^{c,0},\ldots,\bar{J}_R^{c,M})^T,
\end{equation*}
respectively, the system can be rewritten as
\begin{equation}
    \partial_t U + A(U)\partial_x U = \dfrac{1}{\epsilon} \bar{J}_R(U) \label{eq_kin:discretized_model_system_fbar}
\end{equation}
with $A(U)\in \mathcal{M}_{N(M+1)\times N(M+1)}(\mathbb{R})$ defined as
\begin{equation*}
    A(U) = 
    \begin{pmatrix}
        V & 0 & \cdots & 0 & A^1(U)\\
        0 & V & \cdots & 0 & A^2(U)\\
        \vdots & \vdots & \ddots & \vdots & \vdots\\
        0 & 0 & \cdots & V & A^{N-1}(U)\\
        0 & 0 & \cdots & 0 & A^N(U) + V
    \end{pmatrix}
\end{equation*}
with $V = v I_{M+1}$, where $v=(v_0,\ldots,v_M)$ is the velocity vector and $I_{M+1}$ is the identity matrix. The submatrices $A^c = A^c(U)$, $c=1,\ldots,N$, are defined as $A^c = (a_{i,j}^c)$, 
where the coefficients $a_{i,j}^c$ are, for $i,j\in\lbrace 0,\ldots,M\rbrace$,
\begin{equation*}
    \begin{array}{ll}
        a_{i,i}^c = -\dfrac{1}{1-r}\displaystyle\sum_{\ell=i+1}^M (v_{\ell}-v_i) \bar{f}_{\ell}^c, & \qquad a_{i,j}^c = \dfrac{1}{1-r} (v_i-v_j) \bar{f}_i^c \;\; \text{for} \;\; j<i, \\
        a_{i,j}^c = 0 \;\; \text{for} \;\; j>i, & \qquad a_{M,M}^c=0.
    \end{array}
\end{equation*}
In particular,
\begin{equation*}
    a_{0,0}^c = - \dfrac{\bar{q}^c}{1-r}.
\end{equation*}

\subsection{Properties of the discrete model for N classes}
\label{sec_kin:discretizationmodel:properties_discr_model}

We will now study the properties of the hyperbolic part of \eqref{eq_kin:discretized_model_system_fbar}. 
\begin{proposition}
    \label{prop_kin:eigenvalues_discrete_Nc_fbar}
    For $N,M\geq 1$, consider the system of equations \eqref{eq_kin:discretized_model_system_fbar}. Then, the roots of the characteristic polynomial associated to the matrix of the system $A(U)$ are
    \begin{equation}
        \begin{array}{lll}
            \lambda_i^c = v_i, & i=0,\ldots,M-1, & c=1,\ldots,N-1,\\
            \lambda_i^N = v_i - \dfrac{1}{1-r} \displaystyle \sum_{\ell=i+1}^M (v_{\ell}-v_i) \bar{f}_{\ell}^N, & i=0,\ldots,M-1,\\
            \lambda_M^c = 1, & c = 1,\ldots,N.
        \end{array}
        \label{eq_kin:eigenvalues_discrete_Nc_fbar}
    \end{equation}
\end{proposition}
\begin{proof}
    It is straightforward to find the corresponding characteristic polynomial using the properties of determinants of block matrices and the fact that $A^N + V$ is lower triangular:
    \begin{align*}
        p_{A(U)}(\lambda) &= \det(A(U) - \lambda I_{N(M+1)})\\
        &= \det
        \begin{pmatrix}
            V-\lambda I_{M+1} & 0 & \cdots & 0 & A^1\\
            0 & V-\lambda I_{M+1} & \cdots & 0 & A^2\\
            \vdots & \vdots & \ddots & \vdots & \vdots\\
            0 & 0 & \cdots & V-\lambda I_{M+1} & A^{N-1}\\
            0 & 0 & \cdots & 0 & A^N + V -\lambda I_{M+1}
        \end{pmatrix}\\
        &= \det(V-\lambda I_{M+1})^{N-1} \cdot \det(A^N + V-\lambda I_{M+1})\\
        &= (1-\lambda)^N \prod_{i=0}^{M-1}(v_i-\lambda)^{N-1} \prod_{i=0}^{M-1}(v_i + a_{i,i}^N - \lambda)\\
        &= (1-\lambda)^N \prod_{i=0}^{M-1}(v_i-\lambda)^{N-1} \prod_{i=0}^{M-1} \left(v_i - \dfrac{1}{1-r} \displaystyle \sum_{\ell=i+1}^M (v_{\ell}-v_i) \bar{f}_{\ell}^N - \lambda\right).
    \end{align*}
\end{proof}
There are multiple eigenvalues depending on the state of the system, so hyperbolicity is not obvious at the present stage. We also observe that, for $i<j$ and $c=1,\ldots,N-1$, $\lambda_i^N\neq \lambda_j^c$. 

Following the single-class case, the natural extension for the physical domain of definition is
\begin{align*}
	\Delta_M^N = \{ \bar{f} = (\bar{f}^1,\ldots,\bar{f}^N) &\text{ with } \bar f ^c = (\bar{f}_0^c,\ldots,\bar{f}_M^c)\in\mathbb{R}^{M+1}, \; c=1,\ldots,N \; |\\
	&0\leq\bar{f}_i^1\leq\cdots\leq\bar{f}_i^N, \; \bar{\rho}^c=\sum_{i=0}^M \bar{f}_i^c, \; \bar{\rho}^N=r\leq 1 \}.
\end{align*}

\begin{proposition}
	\label{propinv}
    Under the condition $\bar{f}\in\Delta_M^N$, the system matrix $A(U)$ is diagonalizable, and therefore the system is hyperbolic.
\end{proposition}

\begin{remark}
	\label{rem_inv1}
Note that the condition $\bar{f}\in\Delta_M^N$ is an assumption and not a property of the solutions of the equation, see Remark \ref{inv2}
\end{remark}

\begin{proof}[Proof of Proposition \ref{propinv}]
    To compute the eigenvectors, one observes directly that the eigenvectors associated to $\lambda_i^c=v_i$, $c=1,\ldots,N-1$, $i=0,\ldots,M$, are given by
    \begin{equation*}
        \mathbf{R}_i^c = (0,\ldots,0,e_i^c,0,\ldots,0)^T\in\mathbb{R}^{N(M+1)},
    \end{equation*}
    with $0\in\mathbb{R}^{M+1}$, and $e_i^c\in\mathbb{R}^{M+1}$ the unit vector associated to class $c$, i.e., the only non-zero element of $\mathbf{R}_i^c$ is in the position $c(i+1)$. The eigenvectors $\mathbf{X}_i=\mathbf{R}_i^N$ associated to the eigenvalues $\lambda_i^N$, $i=0,\ldots,M+1$, are given by
    \begin{equation*}
        \mathbf{X}_i = (x_i^1,\ldots,x_i^N)^T\in\mathbb{R}^{N(M+1)},
    \end{equation*}
    where $x_i^N\in\mathbb{R}^{M+1}$ are the eigenvectors to $A^N + V$ for the eigenvalues $\lambda_i^N$, given by
    \begin{equation*}
        x_i^N = (0,\ldots,0,1,x_i^{N,i+1},\ldots,x_i^{N,j},\ldots,x_i^{N,M})^T,
    \end{equation*}
    with
    \begin{equation*}
        x_i^{N,j} = - \dfrac{\bar{f}_j^N}{1-\sum_{\ell=0}^i \bar{f}_{\ell}^N}
    \end{equation*}
    for $j=i+1,\ldots,M$.  This is the same result as in the one-class case\cite{BK18}. For $c=1,\ldots,N-1$, $x_i^c$ is computed by
    \begin{equation*}
        V x_i^c + A^c x_i^N = \lambda_i^N x_i^c.
    \end{equation*}
    In other words, to guarantee the hyperbolicity of the system, we have to make sure that the linear system
    \begin{equation}
        (\lambda_i^N I_{M+1} - V) x_i^c = A^c x_i^N, \label{eq_kin:linsyst}
    \end{equation}
    with singular matrix $\lambda_i^N I_{M+1} - V$, has a solution. This is observed as follows. Due to the lower triangular form of $A^c$, we have
    \begin{equation}
        A^c x_i^N = \left( 0,\ldots,0, (A^c x_i^N)^i, (A^c x_i^N)^{i+1}, \ldots, (A^c x_i^N)^j, \ldots, (A^c x_i^N)^M \right)^T. \label{eq_kin:linsys_hyperbolicity}
    \end{equation}
    Assuming $\bar{f}\in\Delta_M^N$, we have that the eigenvalue $\lambda_i^N$, $i=0,\ldots,M$, satisfies
    \begin{equation*}
    	\lambda_i^N = v_i - \dfrac{1}{1-r} \sum_{\ell=i+1}^M (v_{\ell}-v_i) \bar{f}_{\ell}^N \leq v_i.
    \end{equation*}
    It immediately follows that $\lambda_i^N<v_j$ for all $j>i$. However, 0 can only be one of the first $i$ diagonal elements of the matrix $\lambda_i^N I_{M+1} - V$. Therefore, in case $\lambda_i^N<v_i$, a solution of \eqref{eq_kin:linsyst} is given by
    \begin{equation*}
        x_i^c = \left( 0,\ldots,0, \frac{(A^c x_i^N)^i}{\lambda_i^N-v_i}, \frac{(A^c x_i^N)^{i+1}}{\lambda_i^N-v_{i+1}}, \ldots, \frac{(A^c x_i^N)^j}{\lambda_i^N-v_j}, \ldots, \frac{(A^c x_i^N)^M}{\lambda_i^N-v_M} \right)^T.
    \end{equation*}
    with
    \begin{equation*}
        x_i^{c,j} =
        \begin{cases}
            0 & \text{ for } j<i,\\
            \dfrac{a_{i,i}^c}{\lambda_i^N-v_i} & \text{ for } j=i,\\[3pt]
            \dfrac{1}{\lambda_i^N-v_j} \left( a_{j,i}^c + \displaystyle\sum_{k=i+1}^j a_{j,k}^c x_i^{N,k}  \right) & \text{ for } j=i+1,\ldots,M.
        \end{cases}
    \end{equation*}
    The case $\lambda_i^N=v_i$ requires that $\bar{f}^N_{\ell}=0$, $\ell>i$. Since by assumption $\bar{f}\in\Delta_M^N$, we have that, for all $c$, one obtains $\bar{f}_{\ell}^c=0$, $\ell>i$. Thus, in case $\lambda_i^N=v_i$, since every $\bar{f}_{\ell}^c$, $\ell>i$, is equal to zero, the $i$-th component of $A^c x_i^N$, i.e., $a_{i,i}^c$, is also zero. Choosing in this case the $i$-th component of the solution $x_i^c$ to be equal to $0$ gives again a valid solution.
    
    We will now prove that the full set of eigenvectors
    \begin{equation*}
        \{ \mathbf{R}_i^1,\ldots, \mathbf{R}_i^{N-1}, \mathbf{X}_i, \; i=0,\ldots,M\}
    \end{equation*}
    is linearly independent. Consider the combination
    \begin{equation*}
        \sum_{c=1}^{N-1} \sum_{i=0}^M \alpha_i^c \mathbf{R}_i^c + \sum_{i=0}^M \beta_i \mathbf{X}_i = 0.
    \end{equation*}
    For each $i$, the unit eigenvectors $\mathbf{R}_i^c$, $c<N$, have zero entries in the $N$-th block, while $\mathbf{X}_i$ contains the vector $x_i^N$. Therefore, for block $N$, $\sum_i \beta_i x_i^N = 0$. The $x_i^N$ are linearly independent by \cite{BK24discrete}, so $\beta_i=0$ for every $i$. By definition of the eigenvectors $\mathbf{R}_i^c$, this directly gives $\alpha_i^c=0$ for all $i$ and $c$. Thus, the full set of $N(M+1)$ eigenvectors is linearly independent and the system is hyperbolic.
\end{proof}

\begin{remark}
	\label{rem_kin:linear_degeneracy}
	The total linear degeneracy of the problem is obvious due to the 1-class case, and the fact that the eigenvalues $\lambda_i^N$ depend only on $\bar{f}^N$ and not on $\bar{f}^c$, $c\neq N$, and that the $\lambda_i^c$, $c\neq N$, are constant. Therefore, for each $(c,i)$-characteristic field, $c=1,\ldots,N$, $i=0,\ldots,M$, the Hugoniot loci and the integral curves associated to the shock and rarefaction waves, respectively, coincide, and thus the Lax curves are given by contact discontinuities.
\end{remark}

\begin{ex}
    \label{ex_kin:discrete_model_N=2,M=1}
    $N=2$, $M=1$. For the 2-class, 2-velocity model with velocities $0=v_0<v_1=1$, we obtain the model
    \begin{align*}
        \partial_t \bar{f}_0^1 &- \dfrac{1}{1-r} \bar{f}_1^1 \partial_x \bar{f}_0^2 = 0,\\
        \partial_t \bar{f}_1^1 &+ \partial_x \bar{f}_1^1 + \dfrac{1}{1-r} \bar{f}_1^1 \partial_x \bar{f}_0^2 = 0,\\
        \partial_t \bar{f}_0^2 &- \dfrac{1}{1-r} \bar{f}_1^2 \partial_x \bar{f}_0^2 = 0,\\
        \partial_t \bar{f}_1^2 &+ \partial_x \bar{f}_1^2 + \dfrac{1}{1-r} \bar{f}_1^2 \partial_x \bar{f}_0^2 = 0.
    \end{align*}
    The system matrix is
    \begin{equation*}
        A =
        \begin{pmatrix}
            V & A^1\\
            0 & A^2+V
        \end{pmatrix}, \quad \text{with} \quad
        V =
        \begin{pmatrix}
            0 & 0\\
            0 & 1
        \end{pmatrix} \quad \text{and} \quad
        A^c =
        \begin{pmatrix}
            -\frac{\bar{f}_1^c}{1-r} & 0\\
            \frac{\bar{f}_1^c}{1-r} & 0
        \end{pmatrix}
        =
        \begin{pmatrix}
            -\frac{\bar{q}^c}{1-r} & 0\\
            \frac{\bar{q}^c}{1-r} & 0
        \end{pmatrix}.
    \end{equation*}
    The eigenvalues of $A^2+V$ are
    \begin{equation*}
        \lambda_0^2 = -\dfrac{\bar{q}^2}{1-r}, \qquad \lambda_1^2=1
    \end{equation*}
    with associated eigenvectors
    \begin{equation*}
        x_0^2 = 
        \begin{pmatrix}
            1\\
            -\frac{\bar{q}^2}{1-\bar{f}_0^2}
        \end{pmatrix},
        \qquad x_1^2 = 
        \begin{pmatrix}
            0\\
            1
        \end{pmatrix}.
    \end{equation*}
    Then,
    \begin{equation*}
        A^1 x_0^2 = \frac{\bar{q}^1}{1-r} 
        \begin{pmatrix}
            -1\\
            1
        \end{pmatrix},
        \qquad
        A^1 x_1^2 = 
        \begin{pmatrix}
            0\\
            0
        \end{pmatrix}.
    \end{equation*}
    For $x_0^1$,
    \begin{equation*}
        (\lambda_0^2 I_2 - V) x_0^1 = 
        \begin{pmatrix}
            -\frac{\bar{q}^2}{1-r} & 0\\
            0 & -1-\frac{\bar{q}^2}{1-r}
        \end{pmatrix} x_0^1 = A^1 x_0^2,
    \end{equation*}
    obtaining
    \begin{equation*}
        x_0^1 = \dfrac{\bar{q}^1}{\bar{q}^2}
        \begin{pmatrix}
            1\\
            -\frac{\bar{q}^2}{1-\bar{f}_0^2}
        \end{pmatrix} = \dfrac{\bar{q}^1}{\bar{q}^2} x_0^2,
    \end{equation*}
    and the eigenvector $\mathbf{X}_0 = \mathbf{R}_0^2 \in\mathbb{R}^4$ would be
    \begin{equation*}
        \mathbf{X}_0 =
        \begin{pmatrix}
            \frac{\bar{q}^1}{\bar{q}^2} x_0^2\\
            x_0^2
        \end{pmatrix}.
    \end{equation*}
    For the special case when $\bar{q}^2=\bar{f}_1^2=0$, we also obtain $\bar{q}^1=0$ due to $0\leq \bar{f}_1^1\leq \bar{f}_1^2$, and solvability with value 0 and full eigenvector $\mathbf{X}_0=(0,0,1,0)^T$. For $x_1^1$,
    \begin{equation*}
        (\lambda_1^2 I_2-V) x_1^1 = 
        \begin{pmatrix}
            -1 & 0\\
            0 & 0
        \end{pmatrix} x_1^1 = A^1 x_1^2,
    \end{equation*}
    giving, for example, $x_1^1 = (0,0)^T$, and the full eigenvector as $\mathbf{X}_1=\mathbf{R}_1^2=(0,0,0,1)^T$. Together with eigenvectors $\mathbf{R}_0^1=(1,0,0,0)^T$ and $\mathbf{R}_1^1=(0,1,0,0)^T$ associated to $\lambda_0^1=0$ and $\lambda_1^1=1$, respectively, we have hyperbolicity of the 2-class model.
\end{ex}

\begin{remark}
	\label{inv2}
    We note that, for the multi-class case, the domain $\Delta_M^N$ is not invariant. Since our system is linearly degenerate, we can apply Hoff's invariant region theory \cite{Hoff} (specifically Theorem 3.2(a)). To show that $\Delta_M^N$ is not invariant, it suffices to show that the condition fails for the 2-class, 2-velocity case in Example \ref{ex_kin:discrete_model_N=2,M=1}. The associated left eigenvectors, in this case, are
    \begin{equation*}
        \begin{array}{ll}
            l_0^1 = \left( 1,0,-\frac{\bar{f}_1^1}{\bar{f}_1^2},0 \right),\qquad & l_1^1 = \left( 0,1,\frac{\bar{f}_1^2}{1-\bar{f}_0^2},0 \right),\\[8pt]
            l_0^2 = (0,0,1,0),\qquad & l_1^2 = \left( 0,0,\frac{\bar{f}_1^2}{1-\bar{f}_0^2},1 \right).
        \end{array}
    \end{equation*}
    By evaluating these eigenvectors on the boundaries, we find that the moving components ($v_1=1$) and the maximum density limits satisfy Hoff's condition:
    \begin{itemize}
        \item[-] On $\bar{f}_1^1=0$, with normal $n=(0,1,0,0)$, we have $l_1^1=(0,1,0,0)$.
        \item[-] On $\bar{f}_1^1=\bar{f}_1^2$, the difference $l_1^2-l_1^1$ yields a parallel to the normal $n=(0,-1,0,1)$.
        \item[-] On $r=1$, i.e., $\bar{f}_0^2+\bar{f}_1^2=1$, the left eigenvector is $l_1^2=(0,0,1,1)$, which is parallel to the norm of the boundary. 
    \end{itemize}
    Therefore, these boundaries are invariant. The failure of invariance is restricted to the stationary components ($v_0=0$) on the boundaries, $\bar{f}_0^1=0$ and $\bar{f}_0^1=\bar{f}_0^2$, with normals $(1,0,0,0)$ and $(-1,0,1,0)$, respectively. On these boundaries, the left eigenvector $l_0^1$ is only parallel to the first normal, if $\bar{f}_1^1=0$, and only parallel to the second on the plane $\bar{f}_1^1=\bar{f}_1^2$. This does not hold for general traffic states, and therefore the condition is not fulfilled.

    Because the system lacks strict domain invariance for the stationary components, we will enforce these bounds on the numerical results in Section \ref{sec_kin:numericalresults} in order not to obtain unphysical states.
\end{remark}

\subsection{Conservative variables}
\label{sec_kin:discretizationmodel:conservative_variables}

Consider the general equation \eqref{eq_kin:discretized_model_1_fbar} without the relaxation term, that is,
\begin{equation*}
    \partial_t \bar{f}^c_i + v_i \partial_x \bar{f}_i^c - \dfrac{1}{1-r} I_i(\bar{f}^c, \partial_x \bar{f}^N) = 0, \quad i=0,\ldots,M, \quad c = 1,\ldots,N.
\end{equation*}
We do the summation over the $M+1$ equations associated to each class $c$, obtaining
\begin{equation*}
    \partial_t \bar{\rho}^c + \partial_x \bar{q}^c - \dfrac{1}{1-r} \sum_{i=0}^M I_i(\bar{f}^c, \partial_x \bar{f}^N) = 0.
\end{equation*}
The goal is now to see that the last term of the equation is equal to 0, which is straightforward, as
\begin{align*}
    \sum_{i=0}^M I_i(\bar{f}^c, \partial_x \bar{f}^N) &= \sum_{i=0}^{M-1} \sum_{\ell=i+1}^M (v_{\ell}-v_i) \bar{f}_{\ell}^c \partial_x \bar{f}_i^N - \sum_{i=1}^M \sum_{\ell=0}^{M-1} (v_i-v_{\ell}) \bar{f}_i^c \partial_x \bar{f}_{\ell}^N \\
    &= \sum_{i=0}^{M-1} \sum_{\ell=i+1}^M (v_{\ell}-v_i) \bar{f}_{\ell}^c \partial_x \bar{f}_i^N - \sum_{\ell=1}^M \sum_{i=0}^{M-1} (v_{\ell}-v_i) \bar{f}_{\ell}^c \partial_x \bar{f}_i^N\\
    &= 0.
\end{align*}
Therefore, we obtain the classical macroscopic conservation law $\partial_t \bar{\rho}^c + \partial_x \bar{q}^c=0$, $c=1,\ldots,N$. We now rewrite our system, for $c=1,\ldots,N$, as
\begin{align*}
    \partial_t \bar{\rho}^c &+ \partial_x \bar{q}^c=0,\\
    \partial_t \bar{f}^c_i &+ v_i \partial_x \bar{f}_i^c - \dfrac{1}{1-r} I_i(\bar{f}^c, \partial_x \bar{f}^N) = 0, \quad i=1,\ldots,M.
\end{align*}
We are now looking for transformations of the variables that will give a conservative form of the system. Define the elements $N_i^c$, $i=0,\ldots,M$, $c=1,\ldots,N$, as the conservative variables for the above system that depend on the distribution functions $\bar{f}_i^c$. We recall that for class $c=N$, we can apply the results for the single-class formulation of the model. Therefore, for a system with $M+1$ discrete velocities, the conservative variables associated to class $c=N$ are analogous to the ones presented in \cite{BK18}, that is,
\begin{equation*}
    N_0^N = 1-r, \qquad N_i^N = \dfrac{1-r}{1-\sum_{\ell=0}^{i-1}\bar{f}_{\ell}^N}, \qquad i=1,\ldots,M,
\end{equation*}
and the distributions can be reconstructed as
\begin{equation*}
    \bar{f}_M^N = N_0^N \left( \dfrac{1}{N_M^N}-1 \right), \qquad \bar{f}_i^N = N_0^N \left( \dfrac{1}{N_i^N}-\dfrac{1}{N_{i+1}^N} \right), \qquad i=0,\ldots,M-1.
\end{equation*}
However, this formulation cannot be directly extended to the lower classes $c<N$. For these classes, the equations for the microscopic distributions $\bar{f}_i^c$ involve as well several derivative terms of the form $\partial_x \bar{f}_{\ell}^N$. As a consequence, the above transformation does not provide a conservative form for these lower classes.

Each lower class $c<N$ admits two conservative variables, corresponding to the slowest and fastest velocity populations, $v_0=0$ and $v_M=1$,
\begin{equation*}
    N_0^c = \bar{\rho}^c, \qquad N_M^c = \dfrac{1-r}{\bar{f}_M^c}.
\end{equation*}
For $N_0^c$, it is straightforward to see from the macroscopic conservation law that $\partial_t N_0^c + \partial_x \bar{q}^c=0$. To obtain the associated equation written in conservative form for $N_M^c$ one proceeds  as follows:
\begin{align*}
    \partial_t N_M^c &= \partial_t \left(\dfrac{1-r}{\bar{f}_M^c}\right)\\
    &= - \dfrac{1}{\bar{f}_M^c} \partial_t r - \dfrac{1-r}{(\bar{f}_M^c)^2} \partial_t \bar{f}_M^c\\
    &= \dfrac{1}{\bar{f}_M^c} \partial_x \bar{q}^N + \dfrac{1-r}{(\bar{f}_M^c)^2} \left( \partial_x \bar{f}_M^c + \dfrac{\bar{f}_M^c}{1-r} \sum_{\ell=0}^M (1-v_{\ell}) \partial_x \bar{f}_{\ell}^N \right)\\
    &= \dfrac{1}{\bar{f}_M^c} \partial_x \bar{q}^N + \dfrac{1-r}{(\bar{f}_M^c)^2} \left( \partial_x \bar{f}_M^c + \dfrac{\bar{f}_M^c}{1-r} (\partial_x r - \partial_x \bar{q}^N) \right)\\
    &= \dfrac{1}{\bar{f}_M^c} \partial_x r + \dfrac{1-r}{(\bar{f}_M^c)^2} \partial_x \bar{f}_M^c\\
    &= - \partial_x N_M^c.
\end{align*}
Therefore, the equation is $\partial_t N_M^c + \partial_x N_M^c = 0$. In contrast, the equations corresponding to intermediate velocities $i=1,\ldots,M-1$ cannot, in general, be rewritten in conservative form.

\begin{ex}
    \label{ex_kin:conservative_vbls_M=1}
    When the model involves only two velocities ($M=1$), for $c=1,\ldots,N-1$, the conservative equations can be written as
    \begin{equation}
        \partial_t N_0^c + \partial_x \left( \dfrac{N_0^N}{N_1^c} \right) = 0,\qquad \partial_t N_1^c + \partial_x N_1^c = 0, \label{eq_kin:cons_step_1}
    \end{equation}
    with primitive variables
    \begin{equation*}
        \bar{f}_0^c = N_0^c - \dfrac{N_0^N}{N_1^c}, \qquad \bar{f}_1^c = \bar{q}^c = \dfrac{N_0^N}{N_1^c},
    \end{equation*}
    while for $c=N$ they are
    \begin{equation}
        \partial_t N_0^N + \partial_x \left( N_0^N - \dfrac{N_0^N}{N_1^N} \right) = 0,\qquad \partial_t N_1^N + \partial_x N_1^N = 0, \label{eq_kin:cons_step_2}
    \end{equation}
    with
    \begin{equation*}
        \bar{f}_0^N = 1 - \dfrac{N_0^N}{N_1^N}, \qquad \bar{f}_1^N = \bar{q}^N = N_0^N \left( \dfrac{1}{N_1^N} - 1 \right).
    \end{equation*}
\end{ex}

\section{Path-conservative methods}
\label{sec_kin:path_cons_methods}

Following the analysis in Section \ref{sec_kin:discretizationmodel:conservative_variables}, it is generally not straightforward to rewrite the general model in conservative form, unlike the single-class case \cite{BK24discrete}. This is a direct consequence of the interaction terms $I_i(\bar{f}^c,\partial_x \bar{f}^N)$ in \eqref{eq_kin:interaction_term}, which involve products between the state variables of one vehicle class and the spatial derivatives of another. Because these terms cannot be consolidated into a global flux gradient $\partial_x F(U)$, the system is treated in non-conservative quasilinear form without the relaxation term,
\begin{equation}
    \partial_t U + A(U) \partial_x U = 0, \label{eq:noncons_syst}
\end{equation}
where $U=U(x,t)\in \Omega \subset \mathbb{R}^{N(M+1)}$ is the state variable and $A(U):\Omega \rightarrow \mathbb{R}^{N(M+1) \times N(M+1)}$ denotes the matrix of the system. Since the system is written in non-conservative form, the product $A(U)\partial_x U$ is not well-defined in the sense of distributions for discontinuous solutions. Therefore, we use the path-conservative framework of \cite{Pares06} to provide a consistent definition of weak solutions, where the jump conditions across shocks are determined by the choice of a family of paths $\Phi$. We choose such family of paths as in \cite{DalMaso1995DefinitionAW,LeFloch99NonconservProducts}.
\begin{definition}[G. Dal Maso, P.G. LeFloch, and F. Murat \cite{DalMaso1995DefinitionAW}]
    The family of paths $\Phi:[0,1]\times \Omega \times \Omega \rightarrow \Omega$ is a locally Lipschitz continuous map such that it fulfills the following conditions:
    \begin{itemize}
        \item[(a)] $\Phi (0;U^{-},U^{+}) = U^{-}$ and $\Phi (1;U^{-},U^{+}) = U^{+}$ for every $U^{-},U^{+}\in\Omega$, 

        \item[(b)] for every bounded set $\Omega'\subset\Omega$, there exist constants $k_1,k_2$ such that, for almost every $s\in[0,1]$,
        \begin{equation*}
            |\partial_s\Phi (s;U^{-},U^{+})| \leq k_1 |U^{+}-U^{-}|
        \end{equation*} 
        and 
        \begin{equation*}
            |\partial_s \Phi (s;U^{-}_1,U^{+}_1) - \partial_s \Phi (s;U^{-}_2,U^{+}_2)| \leq k_2 (|U^{-}_1-U^{-}_2| + |U^{+}_1-U^{+}_2|)
        \end{equation*}
        for all $U^{-},U^{-}_1,U^{-}_2,U^{+},U^{+}_1,U^{+}_2 \in \Omega'$.
    \end{itemize}
\end{definition}
The product $A(U)\partial_x U$ can be denoted by $[A(U) \partial_x U]_{\Phi}$, a Borel measure, after choosing $\Phi$, and is defined as
\begin{equation}
    \begin{aligned}
        \langle [A(U(\cdot,t)) \partial_x U(\cdot,t)]_{\Phi}, \phi \rangle &= \int_{\mathbb{R}} A(U(x,t)) \partial_x U(x,t) \phi(x) dx\\
        &+ \sum_{x_D\in D(t)} \left( \int_0^1 A(\Phi(s;U^{-}_D,U^{+}_D)) \partial_s \Phi(s;U^{-}_D,U^{+}_D) ds \right) \phi(x_D) 
    \end{aligned} \label{eq:noncons_product}
\end{equation}
for every test function $\phi\in\mathcal{C}_0^{\infty}(\mathbb{R}^{N(M+1)})$. The states $U_D^{-}=U(x_D^{-},t)$ and $U_D^{+}=U(x_D^{+},t)$ denote the solutions to the left and to the right of the discontinuity at position $x_D$ and time $t$, and $D(t) \in \mathbb{R}$ the set of all positions of the discontinuities at time $t$. 

Considering the interval $I_j=[x_{j-1/2},x_{j+1/2}]$ for every $j$, we obtain the update formula
\begin{equation}
    U_j^{n+1} = U_j^n - \lambda \langle [A(U^n) \partial_x U^n]_{\Phi}, \mathbbm{1}_{I_j} \rangle, \quad \lambda=\Delta t/\Delta x. \label{eq:update_eq_noncons}
\end{equation}
For the numerical approximation of \eqref{eq:noncons_product}, we will rely on the discrete term of the nonconservative product, due to the fact that the numerical solution is piecewise constant. Discretizing the product into two terms depending on the contributions of the cells $I_j$ and $I_{j+1}$, i.e.,
\begin{equation}
    \langle [A(U^n) \partial_x U^n]_{\Phi}, \mathbbm{1}_{I_j} \rangle = D_{j-1/2}^{n,+}+D_{j+1/2}^{n,-}, \label{eq:approx_noncons_product}
\end{equation}
gives the desired finite volume scheme, where 
\begin{equation*}
    D^{\pm}:\Omega^{k+l+1}\rightarrow \Omega, \quad D^{n,\pm}_{j+1/2} = D^{\pm}(U_{j-k}^n,\ldots,U_{j+l}^n).
\end{equation*}
In particular, we consider $D^{n,\pm}_{j+1/2} = D^{\pm}(U_j^n,U_{j+1}^n)$ (see \cite{KolbeHerty24}).
\begin{definition}[C. Par\'{e}s \cite{Pares06}]
    The numerical scheme \eqref{eq:update_eq_noncons}-\eqref{eq:approx_noncons_product} is $\Phi$-\textit{conservative} if
    \begin{enumerate}
        \item[(1)] $D^{\pm}(U,U)=0$ for every $U\in\Omega$ and 
        \item[(2)] For each $U^{-},U^{+}\in\Omega$,
        \begin{equation}
            D^{-}(U^{-},U^{+}) + D^{+}(U^{-},U^{+}) = \int_0^1 A(\Phi(s;U^{-},U^{+})) \partial_s \Phi(s;U^{-},U^{+}) ds. \label{eq:def_D_pm}
        \end{equation}
    \end{enumerate}
\end{definition}
As presented in \cite{Pares09}, there are several choices for path-conservative schemes (e.g., Roe method \cite{Roe1981,Pares04Roe}, Godunov method \cite{Godunov,MRP07Godunov}). Hereafter, the numerical flux is set as a generalized Lax-Friedrichs scheme \cite{Lax1954LFscheme,Pares09}, with
\begin{equation*}
    D^{\pm}(U^{-},U^{+}) = \dfrac{1}{2} \int_0^1 A(\Phi(s;U^{-},U^{+})) \partial_s \Phi(s;U^{-},U^{+}) ds \pm \dfrac{1}{2\lambda} (U^{+} - U^{-}).
\end{equation*}
Based on the theoretical framework of \cite{Toumi92}, we have the following definition:
\begin{definition}[I. Toumi \cite{Toumi92}]
    A matrix $\hat{A}(U^{-},U^{+})$ is a \textit{generalized Roe matrix} if it satisfies the following conditions:
    \begin{itemize}
        \item[(1)] \textbf{Consistency:} $\hat{A}(U,U)=A(U)$ for every $U\in\Omega$,
        \item[(2)] \textbf{Hyperbolicity:} $\hat{A}(U^{-},U^{+})$ is diagonalizable with real eigenvalues for any states $U^{-},U^{+}\in\Omega$, and
        \item[(3)] \textbf{Path-consistency:} for any $U^{-},U^{+}\in\Omega$,
        \begin{equation*}
            \hat{A}(U^{-},U^{+})(U^{+}-U^{-}) = \int_0^1 A(\Phi(s;U^{-},U^{+})) \partial_s \Phi(s;U^{-},U^{+}) ds
        \end{equation*} 
    \end{itemize}
\end{definition}
Let 
\begin{align*}
    I_{j+1/2}^n = \hat{A}_{j+1/2}^n(U_{j+1}^n-U_j^n) &= D_{j+1/2}^{n,-} + D_{j+1/2}^{n,+}\\
    &= \int_0^1 A(\Phi(s;U_j^n,U_{j+1}^n)) \partial_s \Phi(s;U_j^n,U_{j+1}^n) ds,
\end{align*}
where $\hat{A}_{j+1/2}^n=\hat{A}(U_j^n,U_{j+1}^n)$ is taken as a generalized Roe matrix. Thus, we obtain the update equation
\begin{equation}
    U_j^{n+1} = \dfrac{1}{2} (U_{j-1}^n + U_{j+1}^n) - \dfrac{\lambda}{2} (I_{j-1/2}^n + I_{j+1/2}^n). \label{eq:gen_path_cons_LF}
\end{equation}
In the following, the nonconservative product $A(U)\partial_x U$ is defined by means of the straight segment path 
\begin{equation}
    \Phi(s;U^{-},U^{+}) = U^{-} + s(U^{+}-U^{-}), \qquad U^{-},U^{+}\in\Omega, \qquad s\in[0,1]. \label{eq:ex_path_family}
\end{equation}
For the straight segment path, the Roe matrix $\hat{A}_{j+1/2}^n$ corresponds to the integral of the system matrix depending on the path, i.e.,
\begin{equation*}
    \hat{A}_{j+1/2}^n = \int_0^1 A(\Phi(s;U_j^n,U_{j+1}^n)) ds.
\end{equation*}

\begin{remark}
    The choice of the family of paths may influence the convergence of path-consistent numerical schemes toward weak solutions. In particular, for linearly degenerate characteristic fields, if the selected path coincides with a parametrization of the corresponding integral curve connecting left and right states, the convergence error vanishes (see \cite{Castro08}). In the present work, we adopt the straight segment path for simplicity, as seen in \cite{KolbeHerty24,Pares09,Volpert1967}.
\end{remark}

\section{Stability of the continuous kinetic model}
\label{sec_kin:stability_cont_model}

In this section, we derive the macroscopic limit of the multi-class continuous kinetic model \eqref{eq_kin:kinbasic} via a formal asymptotic expansion. Following the classical approach for vehicular traffic established in \cite{Nel99,Sopasakis}, we use a Chapman-Enskog expansion to analyze the structural properties and stability of the resulting macroscopic system.

We use the following notations for the first and second moment of each $f^c$, $c=1,\ldots,N$,
\begin{equation*}
	q^c = \int v f^c(v)dv, \qquad e^c = \int v^2 f^c(v)dv.
\end{equation*}
We obtain the balance laws for the first two moments of $f^c$ by multiplying equation \eqref{eq_kin:kinbasic} with $1$ and $v$ and integrating with respect to $v$, 
\begin{equation}
	\begin{aligned}
	    \partial_t \rho^c &+ \partial_x q^c=0,\\
        \partial_t q^c &+ \partial_x e^c + \frac{1}{1-r} \sum_{d=1}^N C(f^c,\partial_x f^d) = \frac{1}{\epsilon T^c} (F^c(\rho) - q^c),
	\end{aligned}
    \label{eq_kin:balance1_cont}
\end{equation} 
with
\begin{equation*}
	C(f,g) = \iint_{\hat{v}<v} (v-\hat{v})^2 f(v) g(\hat{v}) d\hat{v}dv.
\end{equation*} 
Consider now the system \eqref{eq_kin:balance1_cont}. For each class $c=1,\ldots,N$, we take the Chapman-Enskog expansions of $f^c$, $q^c$ and $e^c$ as
\begin{equation*}
    f^c = f^{e,c} + \mathcal{O}(\epsilon), \qquad q^c = F^c(\rho) + \mathcal{O}(\epsilon), \qquad e^c = E^c(\rho) + \mathcal{O}(\epsilon).
\end{equation*}
Substituting in $C(f^c,\partial_x f^d)$ and applying the chain rule, we obtain
\begin{align*}
    C(f^c,\partial_x f^d) &= \iint_{\hat{v}<v} (v-\hat{v})^2 f^c(v) \partial_x f^d(\hat{v}) d\hat{v} dv\\
    &= \iint_{\hat{v}<v} (v-\hat{v})^2 f^{e,c}(\rho,v) \partial_x f^{e,d}(\rho,\hat{v}) d\hat{v} dv + \mathcal{O}(\epsilon)\\
    &= \sum_{i=1}^N \left( \iint_{\hat{v}<v} (v-\hat{v})^2 f^{e,c}(\rho,v) \partial_{\rho^i} f^{e,d}(\rho,\hat{v}) d\hat{v} dv \right) \partial_x \rho^i + \mathcal{O}(\epsilon)\\
    &= \sum_{i=1}^N C(f^{e,c}, \partial_{\rho^i} f^{e,d}) \partial_x \rho^i + \mathcal{O}(\epsilon).
\end{align*}
Applying the chain rule as well for $q^c$ and $e^c$ gives
\begin{equation*}
    \partial_t q^c = \sum_{i=1}^N \partial_{\rho^i} F^c(\rho) \partial_t \rho^i, \qquad \partial_x q^c = \sum_{i=1}^N \partial_{\rho^i} F^c(\rho) \partial_x \rho^i, \qquad \partial_x e^c = \sum_{i=1}^N \partial_{\rho^i} E^c(\rho) \partial_x \rho^i.
\end{equation*}
Substituting the continuity equation in $\partial_t q^c$,
\begin{equation*}
    \partial_t q^c = - \sum_{i=1}^N \sum_{j=1}^N \partial_{\rho^i} F^c(\rho) \partial_{\rho^j} F^i(\rho) \partial_x \rho^j.
\end{equation*}
Therefore, 
\begin{align}
    q^c = F^c(\rho) - \epsilon T^c\sum_{i=1}^N D^{c,i}(\rho) \partial_x \rho^i, \label{eq_kin:Chapman-Enskog_flux}
\end{align}
with 
\begin{equation}
    D^{c,i}(\rho) = \partial_{\rho^i} E^c(\rho) - \sum_{d=1}^N \left( \partial_{\rho^d} F^c(\rho) \partial_{\rho^i} F^d(\rho) - \dfrac{1}{1-r} C(f^{e,c}, \partial_{\rho^i} f^{e,d}) \right). \label{eq_kin:drift-diffusion_param}
\end{equation}
Plugging this into the continuity equation gives a drift-diffusion equation for each $c=1,\ldots,N$,
\begin{equation*}
    \partial_t \rho^c + \sum_{i=1}^N \partial_{\rho^i} F^c(\rho) \partial_x \rho^i = \epsilon T^c \sum_{i=1}^N \partial_x(D^{c,i}(\rho) \partial_x \rho^i).
\end{equation*}
We have obtained a diffusively-corrected multi-class LWR model. It is written in system form as 
\begin{equation}
    \partial_t \rho + \partial_x F(\rho) = \epsilon T \partial_x (D(\rho) \partial_x \rho), \label{eq_kin:drift-diffusion_eq}
\end{equation}
with $T=\operatorname{diag}(T^1,\ldots,T^N)$, $D(\rho)=(D^{c,i}(\rho))$ denoting the diffusion matrix of size $N\times N$ and $F(\rho) = (F^1(\rho),\ldots,F^N(\rho))$ the vector containing all the class-specific fundamental diagrams.

We now investigate  the stability of the multi-species system of equations \eqref{eq_kin:drift-diffusion_eq}. A similar analysis is present in \cite{DCMCLWR}, where a diffusively-corrected multi-class LWR model is discussed. For one class of vehicles such a discussion is described in \cite{Nel00,Nel02}. Choosing a small perturbation $\tilde{\rho}$ around an initial constant state $\rho_0$ such that $\rho(x,t)=\rho_0 + \tilde{\rho}(x,t)$, and substituting into \eqref{eq_kin:drift-diffusion_eq}, yields
\begin{equation}
    \partial_t \tilde{\rho} + \partial_x F(\rho_0 + \tilde{\rho}) = \epsilon T \partial_x (D(\rho_0 + \tilde{\rho}) \partial_x \tilde{\rho}). \label{eq_kin:drift-diffusion_eq:perturbation}
\end{equation}
For the fundamental diagram $F$, we apply a Taylor expansion up to order $\mathcal{O}(\tilde{\rho}^2)$, i.e., 
\begin{equation*}
    F(\rho_0 + \tilde{\rho}) = F(\rho_0) + J(\rho_0) \tilde{\rho} + \mathcal{O}(\tilde{\rho}^2),
\end{equation*}
where $J$ denotes the Jacobian of $F$. Similarly, taking a Taylor expansion of $D$ around $\rho_0$, that is, $D(\rho_0 + \tilde{\rho}) = D(\rho_0) + \mathcal{O}(\tilde{\rho})$, applying it to \eqref{eq_kin:drift-diffusion_eq:perturbation} and neglecting terms of higher order, gives
\begin{equation}
    \partial_t \tilde{\rho} + J(\rho_0) \partial_x \tilde{\rho} = \epsilon T D(\rho_0) \partial_{xx} \tilde{\rho}. \label{eq_kin:drift-diffusion_eq:linearized}
\end{equation}
To investigate  the stability of the solution, we let the perturbation $\tilde{\rho}$ behave like a spatial wave with frequency $\xi\in\mathbb{R}^{+}$ and amplitude $z(t)$ that changes over time, $\tilde{\rho}(x,t) = z(t) e^{i\xi x}$. Substituting  this in \eqref{eq_kin:drift-diffusion_eq:linearized} transforms  the PDE system into 
\begin{equation*}
    z'(t) = -\xi^2 M(\rho_0,\xi) z(t), \qquad M(\rho_0,\xi) = \epsilon T D(\rho_0) + \dfrac{i}{\xi} J(\rho_0) \in\mathbb{C}^{N\times N} \label{eq_kin:drift-diffusion_eq:ODE}
\end{equation*}
with explicit solution $z(t) = e^{-\xi^2 M(\rho_0,\xi) t} z(0)$. For the linearized stability analysis of the multi-species system of equations \eqref{eq_kin:drift-diffusion_eq}, let $M=M(\rho_0,\xi)$ have eigenvalues $\lambda_i^M$ with multiplicity $m_i$, $i=1,\ldots,r$, $m_1+\cdots+m_r=N$. Considering the Jordan canonical form of $M$, the matrix exponential decomposes into independent blocks corresponding to each distinct eigenvalue $\lambda_j^M$ and its associated Jordan matrix $J_j^M$ of size $m_j$. For these Jordan blocks, the structure of $e^{-\xi^2 J_j^M t}$ introduces time-dependent polynomial terms $t^k e^{-\xi^2 \lambda_j^M t}$, $k\leq m_j-1$. Decomposing each eigenvalue into its real and imaginary components, the perturbation modes evolve according to $t^k e^{-\xi^2 \text{Re}(\lambda_j^M)t}$.

For trivial Jordan blocks, i.e., eigenvalues with multiplicity $m_j=1$, the highest polynomial order is $k=0$. In this case, requiring $\text{Re}(\lambda_j^M)\geq 0$ is sufficient to ensure boundedness as $t\to \infty$, with $\text{Re}(\lambda_j^M)=0$ yielding neutral stability with bounded oscillations, and $\text{Re}(\lambda_j^M)>0$ yielding exponential decay. Conversely, if $m_j>1$, there is an explicit dependence on $t^k$. If $\text{Re}(\lambda_j^M)=0$, this term is unbounded for $t\to \infty$. Therefore, strict positivity $\text{Re}(\lambda_j^M)>0$ is required to guarantee that there is an exponential decay for non-trivial Jordan blocks.

For the cases where $F=0$, i.e., $\partial_t \rho = \epsilon T \partial_x(D(\rho)\partial_x\rho)$, it suffices to observe if the real part of the eigenvalues of $D$ is positive, i.e., the matrix is normally elliptic  \cite{Amann}. Further results regarding similar multi-species systems with cross-diffusion can be seen in \cite{ChenJungel21,Jungel2016}.
\begin{remark} \label{rem_kin:stability_N2}
    Consider \eqref{eq_kin:drift-diffusion_eq} for two classes of vehicles ($N=2$). The two eigenvalues of $M$, $\lambda_1^M$ and $\lambda_2^M$, must satisfy the characteristic equation
    \begin{equation*}
        \lambda^2 - \text{tr}(M) \lambda + \det(M) = 0,
    \end{equation*}
    where $\lambda_1^M+\lambda_2^M=\text{tr}(M)$ and $\lambda_1^M \lambda_2^M = \det(M)$. As $M\in\mathbb{C}^{N\times N}$, both the trace and the determinant can be written as $\text{tr}(M)=T_R + iT_I$ and $\det(M)=D_R + iD_I$, respectively. The terms $T_R,T_I,D_R$ and $D_I$ are calculated using the properties of traces and determinants of (complex) matrices, and are defined as
    \begin{equation*}
        \begin{array}{ll}
            T_R = \epsilon \text{tr}(TD), \qquad & D_R = \epsilon^2 \det(TD) - \dfrac{1}{\xi^2} \det(J),\\
            \, T_I = \dfrac{1}{\xi} \text{tr}(J), \qquad & \, D_I = \dfrac{\epsilon}{\xi} \left( \text{tr}(TD) \text{tr} (J) - \text{tr}(TDJ) \right).
        \end{array}
    \end{equation*}
    Given that stability requires $\text{Re}(\lambda_i^{M})\geq 0$, $i=1,2$, the boundary of the stability region is defined by $\text{Re}(\lambda)=0$. Let $\lambda = i\mu$, $\mu\in\mathbb{R}$. Then, substituting $\lambda$ in the previous equation, we obtain 
    \begin{equation*}
        -\mu^2 - i(T_R + iT_I)\mu + (D_R + iD_I) = 0 \implies (-\mu^2 + T_I\mu + D_R) + i(D_I - T_R\mu) = 0.
    \end{equation*}
    Setting the real and imaginary parts to 0, and solving the resulting system of equations for $\mu$ gives $T_R^2 D_R + T_R T_I D_I - D_I^2=0$. Therefore, the system is stable if
    \begin{equation*}
        T_R > 0 \quad \text{and} \quad T_R^2 D_R + T_R T_I D_I - D_I^2 >0.
    \end{equation*}
    In terms of the original matrices, the stability conditions for the 2-class model reads
    \begin{equation}
        \begin{aligned}
            S_1 &= \text{tr}(TD) > 0,\\
            S_2 = S_2(\epsilon) &= \epsilon^2 \text{tr}(TD)^2\det(TD)\\ 
            &\quad\;- \dfrac{1}{\xi^2} \left( \text{tr}(TD)^2 \det(J) - \text{tr}(TD) \text{tr}(J) \text{tr}(TDJ) + \text{tr}(TDJ)^2 \right)>0.
        \end{aligned}
        \label{eq_kin:stability_criterion}
    \end{equation}
    If $T_R = 0$, then $\text{Re}(\lambda_1^M) = -\text{Re}(\lambda_2^M)$, and the stability criterion above does not apply. In this regime, the system is neutrally stable only if $\text{Re}(\lambda_1^M) = \text{Re}(\lambda_2^M) = 0$; otherwise, one eigenvalue necessarily has a negative real part, and thus the system is unstable. Furthermore, if $M$ has a double eigenvalue $\lambda^M$ with $\text{Re}(\lambda^M) = 0$, the system is unstable as well.
\end{remark}

\begin{ex}
    \label{ex_kin:example_stability_cond_cont}
    Consider the equilibrium functions $f^{e,c}$ described in Example \ref{ex_kin:cont_equil_fct}. Then,
    \begin{align*}
        C(f^{e,c}, \partial_{\rho^i} f^{e,d}) =&\; \iint_{\hat{v}<v} f^{e,c}(\rho,v) \partial_{\rho^i} f^{e,d}(\rho,\hat{v}) d\hat{v} dv\\
        =&\; (\rho^c - \alpha_0^c - \alpha_1^c) \partial_{\rho^i} \alpha_0^d \int_0^1 v^2 dv\\ 
        &+ \dfrac{1}{3} (\rho^c - \alpha_0^c - \alpha_1^c) \partial_{\rho^i} (\rho^d - \alpha_0^d - \alpha_1^d) \int_0^1 v^3 dv\\
        &+ \alpha_1^c \partial_{\rho^i} \alpha_0^d + \alpha_1^c \partial_{\rho^i} (\rho^d - \alpha_0^d - \alpha_1^d) \int_0^1 (1-\hat{v})^2 d\hat{v}\\
        =&\ \dfrac{1}{3} (\rho^c - \alpha_0^c - \alpha_1^c) \partial_{\rho^i} \alpha_0^d\\ 
        &+ \dfrac{1}{12} (\rho^c - \alpha_0^c - \alpha_1^c) \partial_{\rho^i} (\rho^d - \alpha_0^d - \alpha_1^d)\\ 
        &+ \alpha_1^c \partial_{\rho^i} \alpha_0^d + \dfrac{1}{3} \alpha_1^c \partial_{\rho^i} (\rho^d - \alpha_0^d - \alpha_1^d).
    \end{align*}
    Substituting the values of the corresponding $\alpha_j^k$ parameters, $j=0,1$, $k=c,d$, we obtain
    \begin{equation*}
        C(f^{e,c}, \partial_{\rho^i} f^{e,d}) = E^c (\partial_{\rho^i} \rho^d - \partial_{\rho^i} F^d) - F^c (\partial_{\rho^i} F^d - \partial_{\rho^i} E^d),
    \end{equation*}
    with,
    \begin{equation*}
        C(f^{e,c}, \partial_{\rho^i} f^{e,d}) = 
        \begin{cases}
            E^c (1 - \partial_{\rho^d} F^d) - F^c (\partial_{\rho^d} F^d - \partial_{\rho^d} E^d) & \text{ if }\; i = d,\\
            - E^c \partial_{\rho^i} F^d - F^c (\partial_{\rho^i} F^d - \partial_{\rho^i} E^d) & \text{ if }\; i\neq d.
        \end{cases}
    \end{equation*}
    Therefore, the diffusion coefficient $D^{i,c}$ is
    \begin{equation*}
        D^{c,i} = \partial_{\rho^i} E^c - \sum_{d=1}^N \left[ \partial_{\rho^d} F^c \partial_{\rho^i} F^d - \dfrac{E^c}{1-r} (\partial_{\rho^i} \rho^d - \partial_{\rho^i} F^d) + \dfrac{F^c}{1-r} (\partial_{\rho^i} F^d - \partial_{\rho^i} E^d) \right].
    \end{equation*}
    Let $E^c(\rho)=F^c(\rho)$, i.e., the class-specific equilibrium functions are given by $$f^{e,c}(\rho,v) = (\rho^c-F^c(\rho)) \delta_0(v) + F^c(\rho) \delta_1(v).$$ This gives the diffusion coefficient
    \begin{equation}
    	\label{eq_kin:general_coefficients_Dci}
        D^{c,i} = \partial_{\rho^i} F^c + \dfrac{F^c}{1-r} - \sum_{d=1}^N \left( \partial_{\rho^d} F^c + \dfrac{F^c}{1-r} \right) \partial_{\rho^i} F^d.
    \end{equation}
    Let
    \begin{equation*}
        B^{c,i} = \partial_{\rho^i} F^c + \dfrac{F^c}{1-r}, \qquad J^{c,i} = \partial_{\rho^i} F^c.
    \end{equation*}
    Then, we can rewrite \eqref{eq_kin:general_coefficients_Dci} as
    \begin{equation}
    		\label{eq_kin:diff_matrix_BJ}
    		D^{c,i} = B^{c,i} - \sum_{d=1}^N B^{c,d} J^{d,i},
    \end{equation}
    and the diffusion matrix can be written as $D = B(I-J)$, where $B=(B^{c,i})$ and $J=(J^{c,i})$.
\end{ex}

\begin{figure}[!ht]
	\centering
	\includegraphics[scale=0.7]{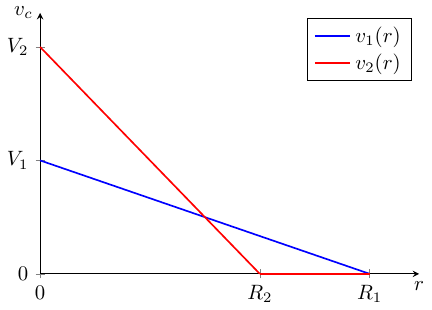}
    	\caption{Velocity function used in Example \ref{ex_kin:stability_cont_model:different_VcRc}.}
	\label{fig:velocityfct:example_diffusion}
\end{figure}

\begin{ex} 
    \label{ex_kin:stability_cont_model:different_VcRc}
    (Example \ref{ex_kin:example_stability_cond_cont} extended) Consider the macroscopic flux
    \begin{equation}
        F^c(\rho)=\rho^c v_c(r), \qquad v_c(r) = V_c\left( 1 - \dfrac{r}{R_c} \right), \label{eq_kin:Fc_Greenshields}
    \end{equation}
    where $v_c(r)$ corresponds to the Greenshields' velocity function \cite{Greenshields1935}. Then, taking $E^c=F^c$,
    \begin{equation*}
        J^{c,i} = v_c(r) + (\partial_{\rho^i} \rho^c) v_c'(r),\qquad B^{c,i} = J^{c,i} + \dfrac{\rho^c v_c(r)}{1-r},
    \end{equation*}
    and the diffusion matrix is
    \begin{equation}
    		\begin{aligned}
    			D^{c,i} =& \; (\partial_{\rho^i} \rho^c) v_c(r) + \rho^c \left( \dfrac{v_c(r)}{1-r} - \dfrac{V_c}{R_c} \right)\\ 
        		&- \sum_{d=1}^N \left[ (\partial_{\rho^d} \rho^c) v_c(r) + \rho^c \left( \dfrac{v_c(r)}{1-r} - \dfrac{V_c}{R_c} \right) \right] \left( (\partial_{\rho^i} \rho^d) v_d(r) - \dfrac{V_d}{R_d} \rho^d \right) \\
        		=& \; (\partial_{\rho^i} \rho^c) v_c(r) (1-v_c(r)) + \rho^c \left[ \dfrac{V_c}{R_c} v_c(r) + C^c(1-v_i(r)+K) \right],
    		\end{aligned}
    		\label{eq_kin:coefficients_Dci}
    \end{equation}
    with
    \begin{equation*}
        C^c = -\dfrac{V_c(1-R_c)}{R_c(1-r)}, \qquad K = \sum_{d=1}^N \dfrac{V_d}{R_d} \rho^d.
    \end{equation*}
    In particular, if $R_k=1$ for some $k\in\lbrace 1,\ldots,N \rbrace$, then $C^k=0$ and
    \begin{equation}
        D^{k,i} = (\partial_{\rho^i} \rho^k) v_k(r) (1-v_k(r)) + V_k\rho^k v_k(r).
        \label{eq_kin:coefficients_Dci_reduced}
    \end{equation}
    Consider now $N=2$ classes (see Figure \ref{fig:velocityfct:example_diffusion}). Let $T^1=T^2=1$ and $\epsilon>0$. Taking $V_1<1=V_2$ and $R_1=1>R_2$, $D=B(I-J)$ would have the reduced coefficients \eqref{eq_kin:coefficients_Dci_reduced} for the first row and the full coefficients \eqref{eq_kin:coefficients_Dci} for the second row, and using the notation in \eqref{eq_kin:diff_matrix_BJ}, $B$ and $J$ are defined as
    \begin{equation*}
    		B(\rho) = 
    		\begin{pmatrix}
    			V_1(1-r) & 0\\
    			\frac{\rho^2}{1-r} \left( 1-\frac{1}{R_2} \right) & 1-\frac{r}{R_2} + \frac{\rho^2}{1-r} \left( 1-\frac{1}{R_2} \right)
    		\end{pmatrix}, \quad
            J(\rho) =
    		\begin{pmatrix}
    			V_1(1-r-\rho^1) & -V_1\rho^1\\
    			-\frac{\rho^2}{R_2} & 1-\frac{r+\rho^2}{R_2}
    		\end{pmatrix}.
    \end{equation*}
\end{ex}

\begin{figure}[!ht]
	\centering
	\includegraphics[scale=0.7]{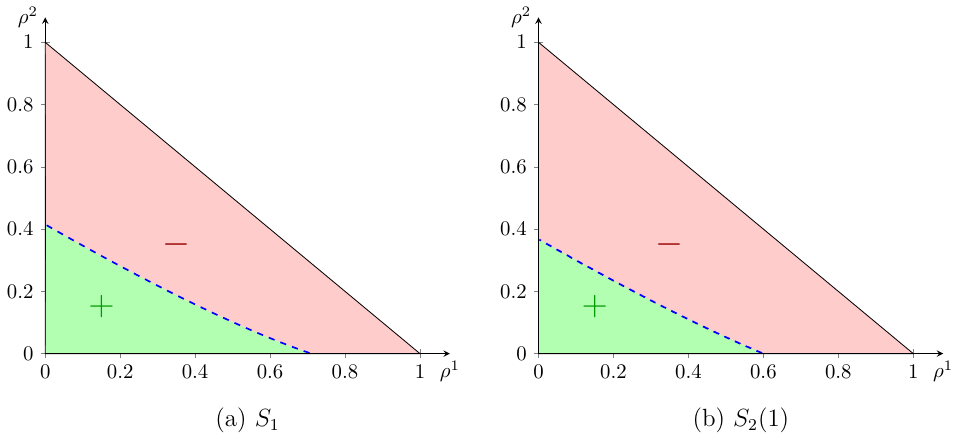}
	\caption{Positive and negative values of $S_1$ and $S_2(1)$ in Example \ref{ex_kin:stability_cont_model:different_VcRc} for $\rho^1,\rho^2$, with $V_1=0.6=R_2$. }
	\label{fig:stability_cond:V1_06_R2_06}
\end{figure}

\begin{figure}[!ht]
	\centering
	\includegraphics[scale=0.63]{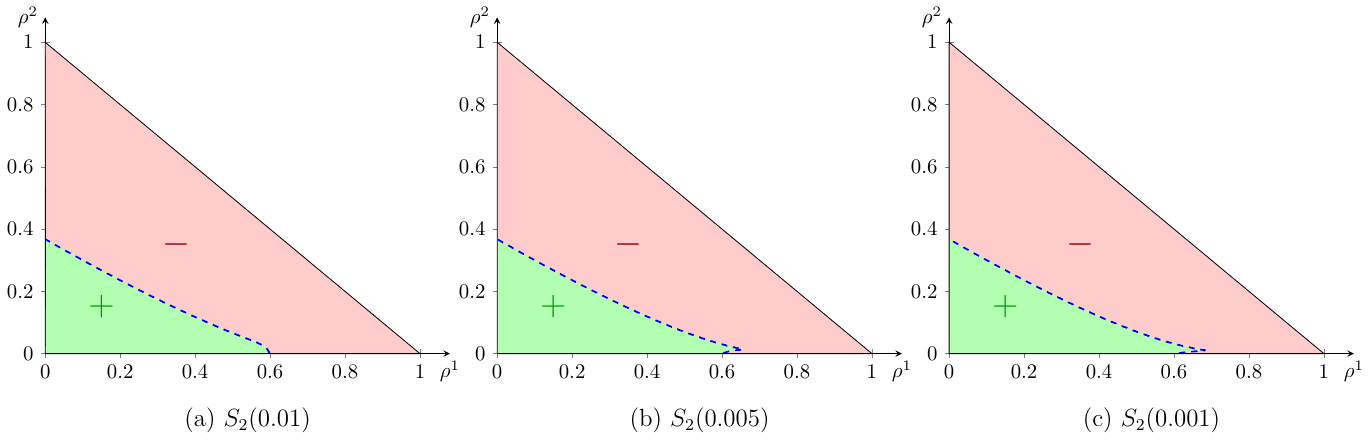}
	\caption{Positive and negative values of $S_2(\epsilon)$, $\epsilon=0.01,0.005,0.001$, in Example \ref{ex_kin:stability_cont_model:different_VcRc} for $\rho^1,\rho^2$, with $V_1=0.6=R_2$.}
	\label{fig:stability_cond:S2_different_eps}
\end{figure}

Figures \ref{fig:stability_cond:V1_06_R2_06} and \ref{fig:stability_cond:S2_different_eps} plot the stability conditions \eqref{eq_kin:stability_criterion} for Example \ref{ex_kin:stability_cont_model:different_VcRc}, where $S_2$ is shown for different values of $\epsilon$. As for this particular choice of fundamental diagram $F^c$ the area where $S_1>0$ fully contains the area where $S_2$ is positive, we will just restrict ourselves to plotting the latter in the numerical tests (see Section \ref{sec_kin:numericalresults:rel_periodic}).

It is important to note that the second stability condition inherently depends on the spatial perturbation frequency $\xi\in\mathbb{R}^{+}$. For a given state to be considered strictly stable, $S_2>0$ must be satisfied for all possible frequencies. Therefore, to generate the stability diagrams, we do not plot the region for a single isolated frequency. Instead, we perform a parameter sweep of $\xi$ over a sufficiently wide logarithmic range (e.g., $[10^{-2}, 10^2]$), where the resulting stable area depicted in the figures represents the global intersection of the stable regimes.

\section{Numerical results}
\label{sec_kin:numericalresults}

In this section, we first describe the numerical scheme used to solve our modified model \eqref{eq_kin:kinbasic_fbar}, and then we present different numerical examples to illustrate it.
			
\subsection{The numerical scheme}
\label{sec_kin:numericalresults:numericalscheme}

We use a splitting scheme solving first the advection part in conservative form and then the relaxation part of the equation. We take an equidistant time discretization with $t^{\nu} = \nu\Delta t$ and a spatial discretization with $x_j=j\Delta x$, where $\nu,j\in \mathbb{N}_0$ denote the time iteration and the cell on the road, respectively. Let $\bar{f}_{i,j}^{c,\nu}$ be the approximation of $\bar{f}_i^c(x_j,t^{\nu})$ for classes $c=1,\ldots,N$ and velocities $i=0,\ldots,M$, with $U_j^{c,\nu} = (\bar{f}_{0,j}^{c,\nu},\ldots,\bar{f}_{M,j}^{c,\nu})^T$, $U_j^{\nu} = (U_j^{1,\nu},\ldots,U_j^{N,\nu})^T$.

Suppose we know $\bar{f}_{i,j}^{c,\nu}$ for each cell on the road. The following algorithm describes the updates from time $t^{\nu}$ to $t^{\nu+1}$. \textit{Step 1} corresponds to the path-conservative method described in Section \ref{sec_kin:path_cons_methods}.
\begin{algorithm}
	\makeatletter
	\renewcommand{\fnum@algocf}{\textbf{\algorithmcfname}}
	\makeatother
	\caption{Update and Relaxation Scheme for $N$ classes and $M+1$ velocities}
	\label{alg:general_N_M}
	\SetAlgoLined
	\DontPrintSemicolon

	\nl \textbf{Step 1}: Choose an $\ell$-point Gauss-Lobatto quadrature, $\ell\geq 3$, to calculate for each $j$ the 2 approximations of the matrix of the system $A$, i.e., $\hat{A}^n_{j-1/2}$ and $\hat{A}^n_{j+1/2}$, and solve for one time-step the update equation \eqref{eq:gen_path_cons_LF}, obtaining $\bar{f}_{i,j}^{c,*}$.\;

	\BlankLine

	\nl \textbf{Step 2}: Define $T^k_{\epsilon}=\epsilon T^k$.\;

	\Indp
    		\nl \textit{2.1}: Update $\bar{f}_{i,j}^{c,*} \longrightarrow f_{i,j}^{c,*}$ using \eqref{eq_kin:fbar_def}.\;

    		\nl \textit{2.2}: For $c=1,\ldots,N$ and $i=0,\ldots,M$, solve the relaxation step with an implicit Euler scheme,
    		\begin{equation*}
        		f_{i,j}^{c,\nu+1} = \frac{T^c_{\epsilon}}{T^c_{\epsilon} + \Delta t} f_{i,j}^{c,*} + \frac{\Delta t}{T^c_{\epsilon} + \Delta t} f_{i,j}^{e,c}(\rho_j^{\nu}).
    		\end{equation*}\;

    		\nl \textit{2.3}: Transform back the variables $f_{i,j}^{c,\nu+1} \longrightarrow \bar{f}_{i,j}^{c,\nu+1}$ using \eqref{eq_kin:fbar_def_initial}.\;
	\Indm
\end{algorithm}
\begin{remark}
	For the 2-velocity case ($M=1$), it is also possible to substitute \textit{Step 1} by the first three steps of the conservative algorithm described in \cite[Section 7.2]{BK24discrete} extended to the multi-class case, following Example \ref{ex_kin:conservative_vbls_M=1}.
\end{remark}

We now illustrate the theoretical results obtained in previous sections with some numerical tests. We will consider two classes of vehicles with associated distribution functions $f^1$ and $f^2$, discretized as $f_0^c,\ldots,f_M^c$, $c=1,2$, with $f_i^c\in[0,1]$, such that
\begin{equation*}
    r = \rho^1 + \rho^2 = \sum_{\ell=0}^M (f^1_{\ell} + f^2_{\ell}),
\end{equation*}
where the initial conditions for the densities and distribution functions are set as
\begin{equation*}
    r_0(x) = r(x,0), \qquad \rho_0^c(x) = \rho^c(x,0), \qquad f_i^c(x,0) = \dfrac{1}{M+1} \rho^c_0(x).
\end{equation*}
For the updates of the distribution functions described in \textit{Step 1} of the previous algorithm, we consider $n_x$ discretization points along a space interval $[a,b]$, and a time step $\Delta t$ given by the corresponding CFL number \cite{CFL1928} on the time interval $[t_0=0,t_f]$, where $n_x$, $[a,b]$ and $t_f$ will change depending on the numerical test. 
Regarding the choice for the velocities, they are taken as equidistant points $v_i=i/M$.

\subsection{Test-case 1: Riemann problem without relaxation}
\label{sec_kin:numericalresults:RP_no_rel}

\begin{figure}[!ht]
	\centering
	\includegraphics[scale=0.62]{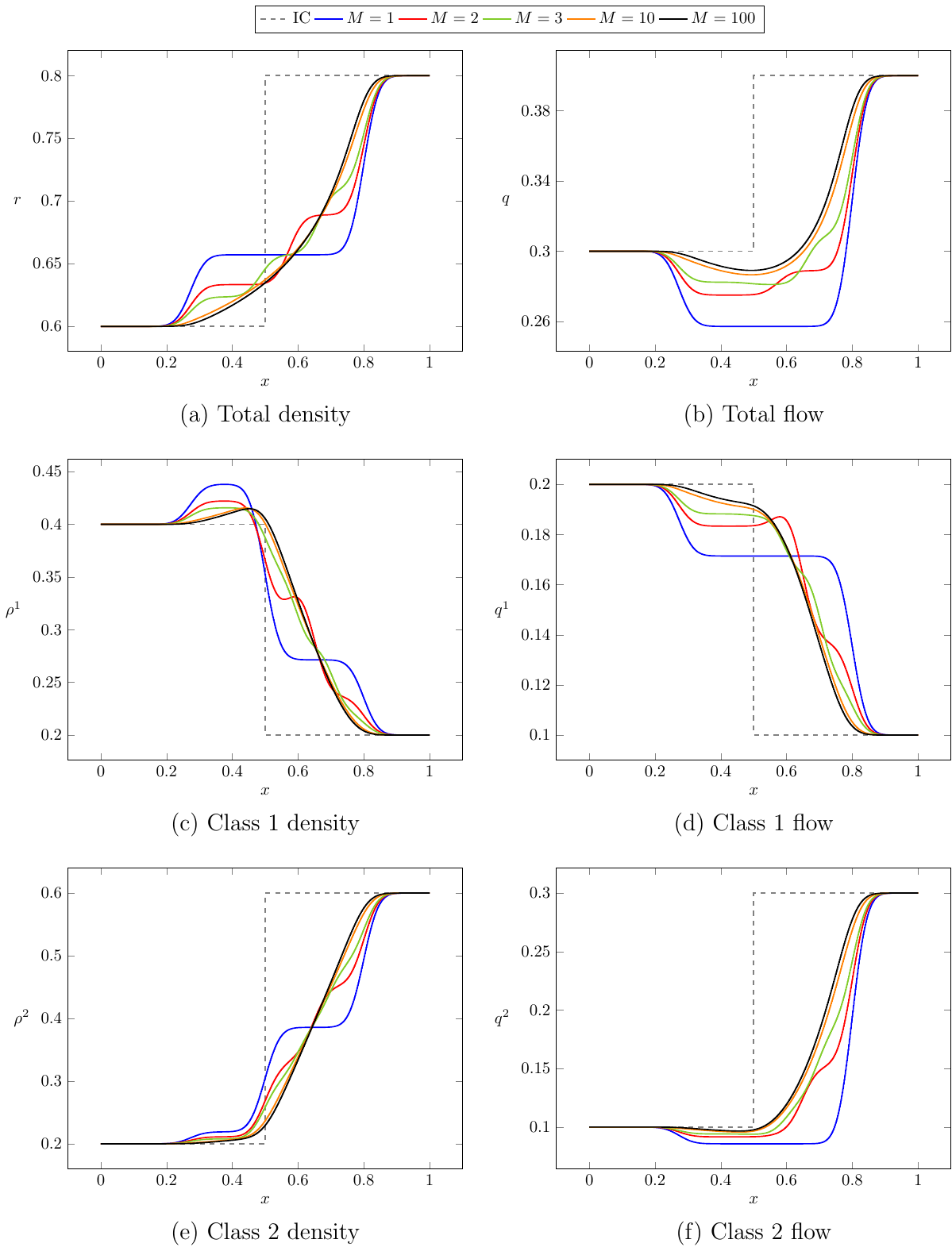}
    \caption{Densities (left column) and flows (right column) for Test-case 1 in Section \ref{sec_kin:numericalresults:RP_no_rel} at time $t_f=0.3$, obtained through \textit{Step 1} of the algorithm for $n_x=1000$ discretization points and different values of $M$. The first row corresponds to the total density $r$ and its corresponding flux $q$. The second and third rows correspond to the classes $c=1$ and $c=2$ of vehicles, respectively. The initial condition is represented with a dashed grey line on each graph.} 
	\label{fig:density_flow:different_M_nx_1000}
\end{figure}

As a first numerical test, we choose a Riemann problem without the relaxation term (i.e., \textit{Step 1} of the algorithm). The initial conditions for the class densities $\rho^1$ and $\rho^2$ are

\begin{equation*}
	\rho^1_0(x) = 
	\begin{cases}
		\rho^1_L = 0.4 & x<0.5,\\
		\rho^1_R = 0.2 & x>0.5,\\
	\end{cases}
	\qquad 
	\rho^2_0(x) = 
	\begin{cases}
		\rho^2_L = 0.2 & x<0.5,\\
		\rho^2_R = 0.6 & x>0.5.\\
	\end{cases}
\end{equation*}
The class-specific fundamental diagram is taken as
\begin{equation*}
		F^c(\rho) = \rho^c(1-r), \qquad c=1,2,
\end{equation*}
so that $q^c=F^c(\rho)$, with a linearly decreasing velocity function $v_c(r)=1-r\in[0,1]$ for $r\in[0,1]$. Because both vehicle classes share identical velocity functions, the multi-class system structurally reduces to the single-class scalar model. Consequently, we expect the evolution of the total density $r(x,t)$ to mirror the dynamics of the standard one-population framework. We consider the time interval $[0,0.3]$ and free boundary conditions.

\begin{figure}[!t]
	\centering
	\includegraphics[scale=0.62]{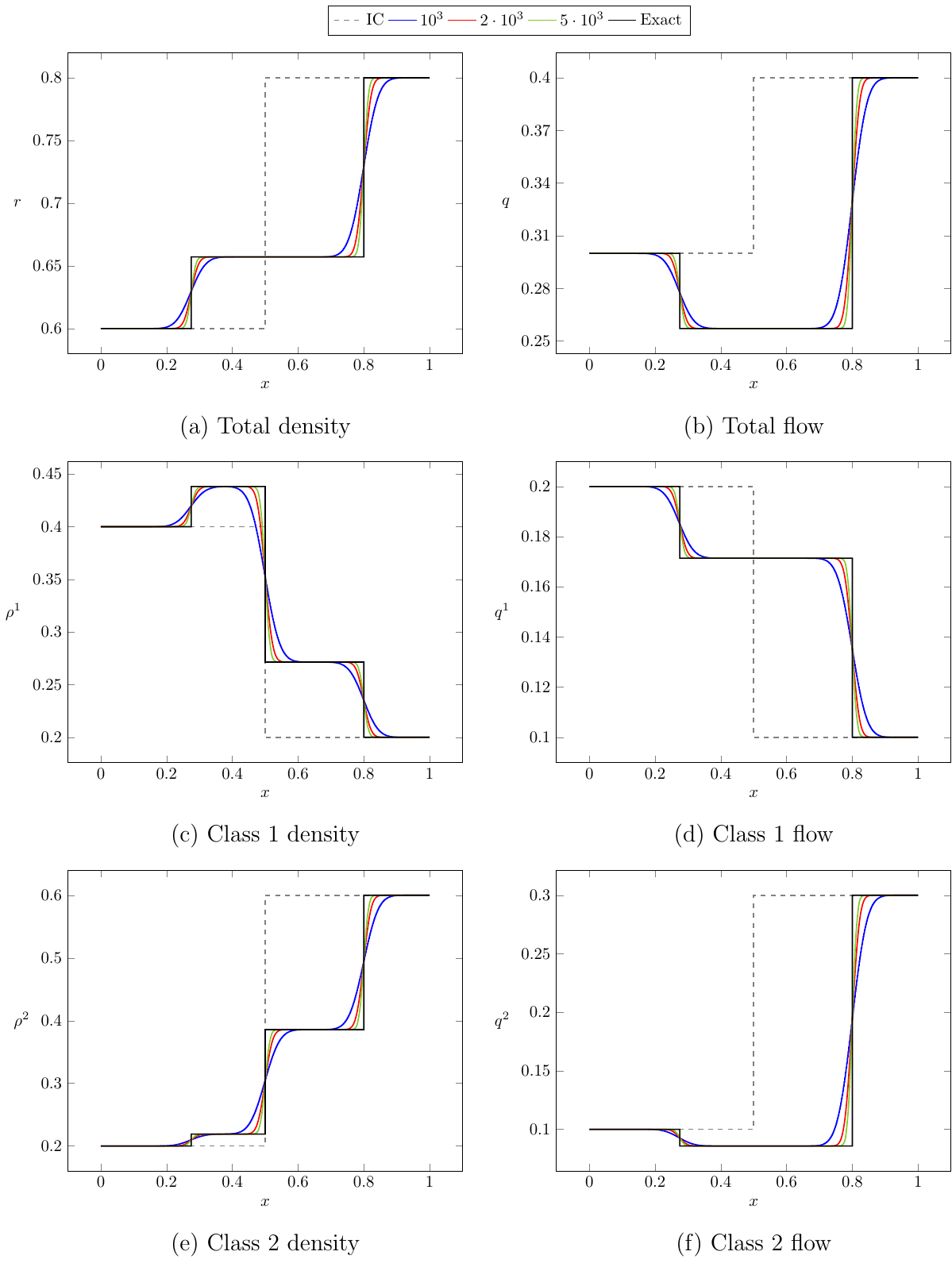}
    \caption{Densities (left column) and flows (right column) for Test-case 1 in Section \ref{sec_kin:numericalresults:RP_no_rel} at time $t_f=0.3$, obtained through \textit{Step 1} of the algorithm for different discretization points and $M=1$. The first row corresponds to the total density $r$ and its corresponding flux $q$. The second and third rows correspond to the classes $c=1$ and $c=2$ of vehicles, respectively. The initial condition is represented with a dashed grey line on each graph. The exact solution for the 2-velocity case is plotted by a black line.} 
	\label{fig:density_flow:M1_different_nx}
\end{figure}

Figure \ref{fig:density_flow:different_M_nx_1000} shows the numerical approximations for the densities and flows at $t_f=0.3$ for different numbers of discrete velocities considering $n_x=1000$. On the other hand, Figure \ref{fig:density_flow:M1_different_nx} shows the comparison between different space discretizations and their convergence to the exact solution for the 2-velocity case.

As seen in Proposition \ref{prop_kin:eigenvalues_discrete_Nc_fbar}, the minimum and maximum characteristic speeds associated with the second modified class $\bar{\rho}^2=r$ for any $M$, i.e., $\lambda_0^2$ and $\lambda_M^2$, correspond exactly to those of the 2-velocity case. Using the initial condition for the $f_i^c$ described above, these extreme eigenvalues simplify to
\begin{equation*}
    \lambda_0^2(r_0) = - \dfrac{r_0}{2(1-r_0)}, \qquad \lambda_M^2=1, \qquad r_0 = r_0(x).
\end{equation*}
Remark \ref{rem_kin:linear_degeneracy} establishes that all characteristic fields are linearly degenerate. Thus, the waves behave as contact discontinuities propagating at constant speeds. In particular, the slowest and fastest wave speeds are determined by $\sigma_{\text{slow}} = \lambda_0^2(r_L) = -0.75$ and $\sigma_{\text{fast}}=1$, respectively, and at time $t_f=0.3$ their corresponding discontinuities are in positions $x_{\text{slow}}=0.275$ and $x_{\text{fast}}=0.8$.

\subsection{Test-case 2: Riemann problem with relaxation}
\label{sec_kin:numericalresults:RP_rel_FanWork}

We now set a different Riemann problem to observe that, when considering the relaxation term described in \textit{Step 2} of the algorithm, we obtain the same behavior as in the $N$-populations model described in \cite{Benzoni-Colombo}. Let the space interval and the time interval be $[0,50]$ and $[0,80]$, respectively. Let the initial conditions for $\rho^1$ and $\rho^2$ be
\begin{equation*}
    \rho_0^1(x) =
    \begin{cases}
        0.5 & x\in[1,10],\\
        0 & \text{else},
    \end{cases}
    \qquad 
    \rho_0^2(x) =
    \begin{cases}
        0.5 & x\in[11,20],\\
        0 & \text{else}.
    \end{cases}
\end{equation*}
In this example, the fundamental diagram is taken as in \eqref{eq_kin:Fc_Greenshields}, with $V_1=1$, $V_2=5/9$ and $R_1=R_2=1$. We take $M=10$ velocities and $n_x=2000$, together with zero inflow and free outflow in the boundaries. For the relaxation, we consider $\epsilon=0.001$ and $T^1=T^2=1$. At the end of the time interval, the vehicles with a higher maximal velocity (e.g., cars) overtake those with a smaller one (e.g., trucks), as seen in Figure \ref{fig:density:overtaking} by $\rho^1$ and $\rho^2$, similarly to the numerical results from  \cite{Benzoni-Colombo,FanWork2015,WongWong2002}. This behavior is consistent with the theoretical framework established in Section \ref{sec_kin:stability_cont_model}. Since the relaxation parameter is taken to be very small ($\epsilon = 0.001$), the kinetic system is driven rapidly toward local equilibrium, effectively recovering the macroscopic multi-class dynamics.

\begin{figure}[!ht]
	\centering
	\includegraphics[scale=0.6]{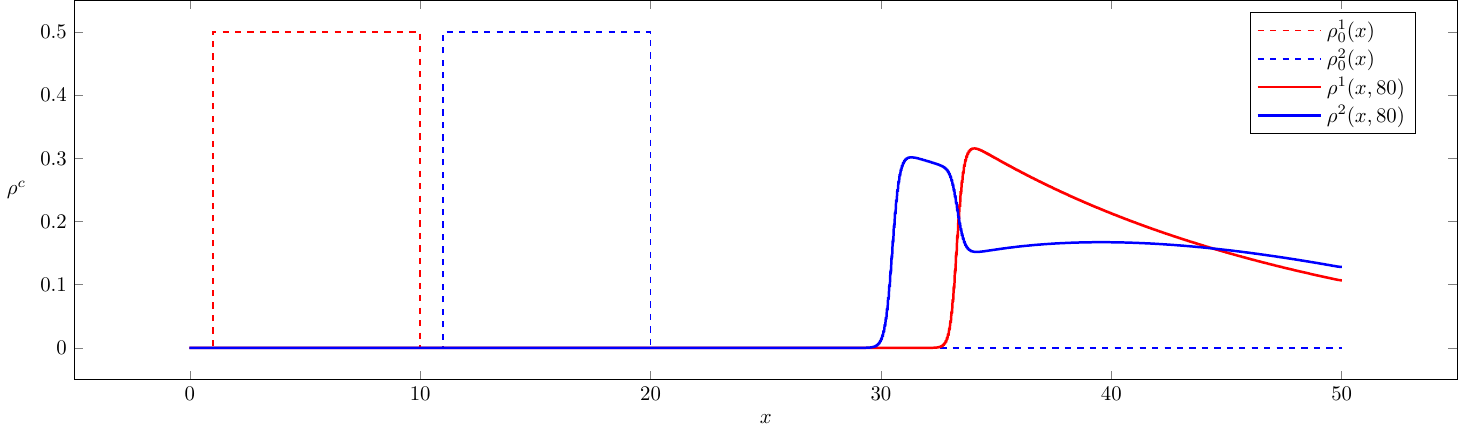}
    \caption{Class densities at time $t_f=80$ for the overtaking example in Test-case 2, Section \ref{sec_kin:numericalresults:RP_rel_FanWork}. The initial condition is represented with a dashed line on each graph.} 
	\label{fig:density:overtaking}
\end{figure}

\subsection{Test-case 3: periodic initial condition with relaxation}
\label{sec_kin:numericalresults:rel_periodic}

In this final test case, we consider the full kinetic system, including the relaxation term, subject to periodic boundary conditions, where the chosen fundamental diagrams follow Example \ref{ex_kin:stability_cont_model:different_VcRc}, with $V_1=0.6=R_2$, $V_2=1=R_1$ (see Figures \ref{fig:velocityfct:example_diffusion} and \ref{fig:stability_cond:V1_06_R2_06}-\ref{fig:stability_cond:S2_different_eps} for the velocity function and the stability regions, respectively). We define the initial density distributions for each vehicle class $c$ as a sinusoidal perturbation around a mean state,
\begin{equation*}
	\rho^c_0(x) = \rho^c_m + k_1 \sin(k_2\pi x),
\end{equation*}
where $k_1$ denotes the amplitude of the perturbation, $k_2$ the spatial frequency and $\rho^c_m$ the spatial average (or mean) of the class-specific density component $\rho^c$ (noted as such to avoid confusion with the cumulative densities $\bar{\rho}^c$). Similarly, let $r_m=\rho^1_m+\rho^2_m$ denote the average of the total density $r=\rho^1+\rho^2$. To measure damping or amplification of traffic waves over time, we define the total density perturbation decay $I(t)$ as 
\begin{equation*}
	I(t) = \int_0^1 |r(x,t)-r_m| dx.
\end{equation*}
For the numerical simulations, we restrict the space domain to the interval $[0,1]$, $M=20$ discrete velocities and $n_x=1000$ discretization points.

\begin{figure}[!t]
	\centering
	\includegraphics[scale=0.48]{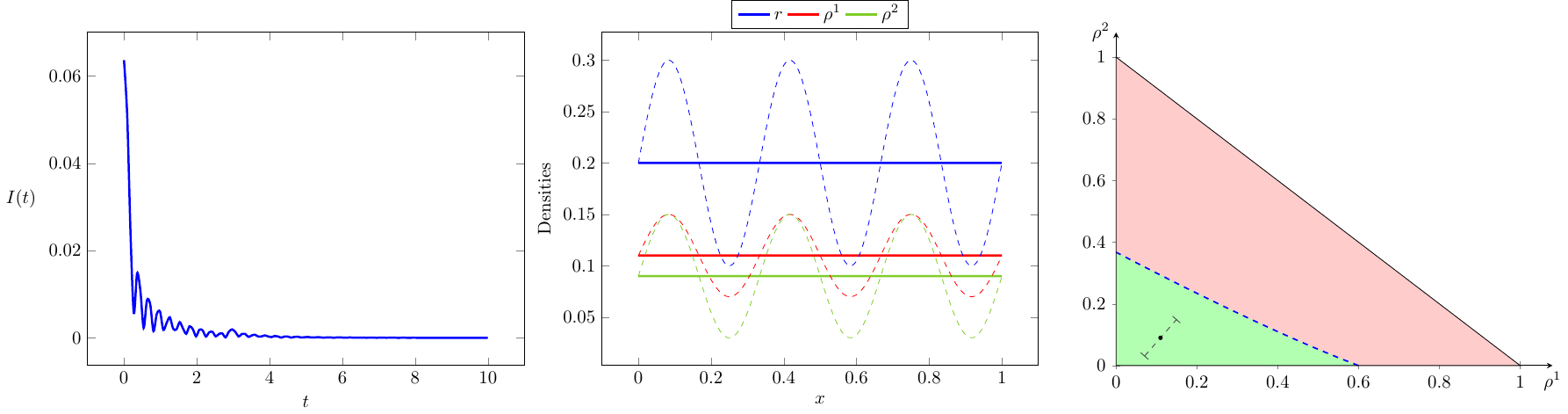}\\
    \includegraphics[scale=0.48]{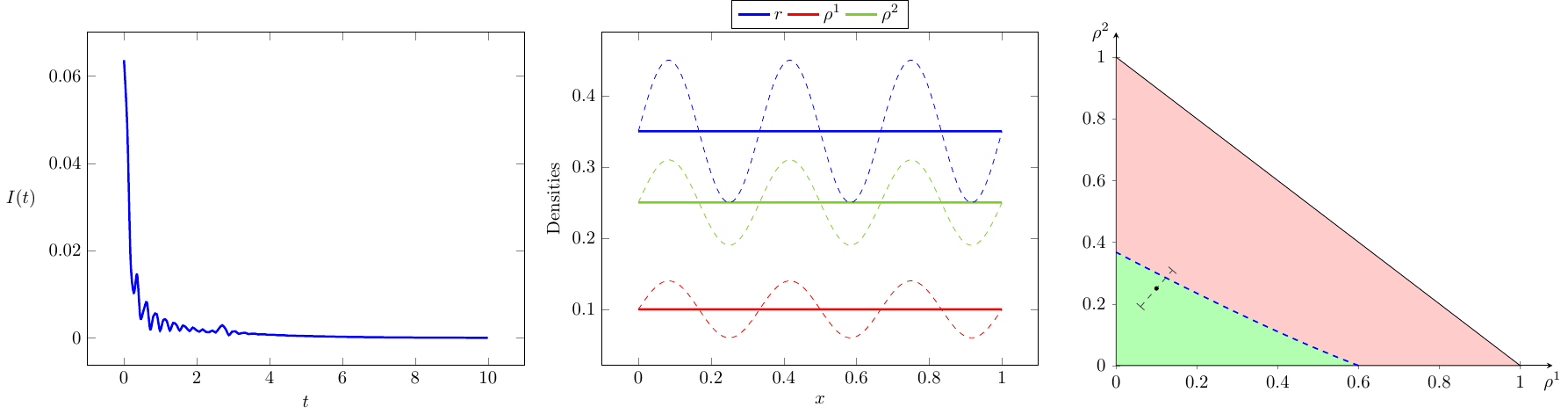}
    \caption{Stable solution for $\epsilon=1$. Plot of decay of solution $I(t)$ (left), densities at the beginning and end of time interval (center), and $\rho_0(x)$, $\rho(x,t_f=10)=\rho_m$, $x\in[0,1]$, represented within the positive and negative areas of $S_2$ (right). For the first row, $\rho_0(x) = (0.11 + 0.09\sin(6\pi x), 0.09 + 0.06\sin(6\pi x))$, while for the second row, $\rho_0(x) = (0.1 + 0.04\sin(6\pi x), 0.25 + 0.06\sin(6\pi x))$.}
	\label{fig:num_test:periodic:stable:eps1}
\end{figure}

\begin{figure}[!ht]
	\centering
    \includegraphics[scale=0.48]{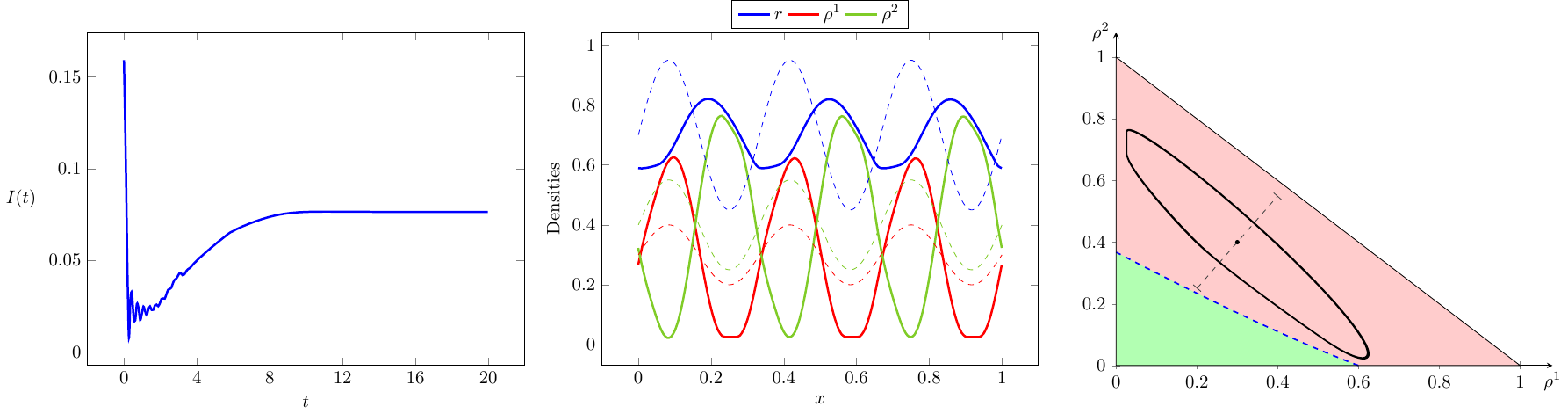}\\
    \includegraphics[scale=0.48]{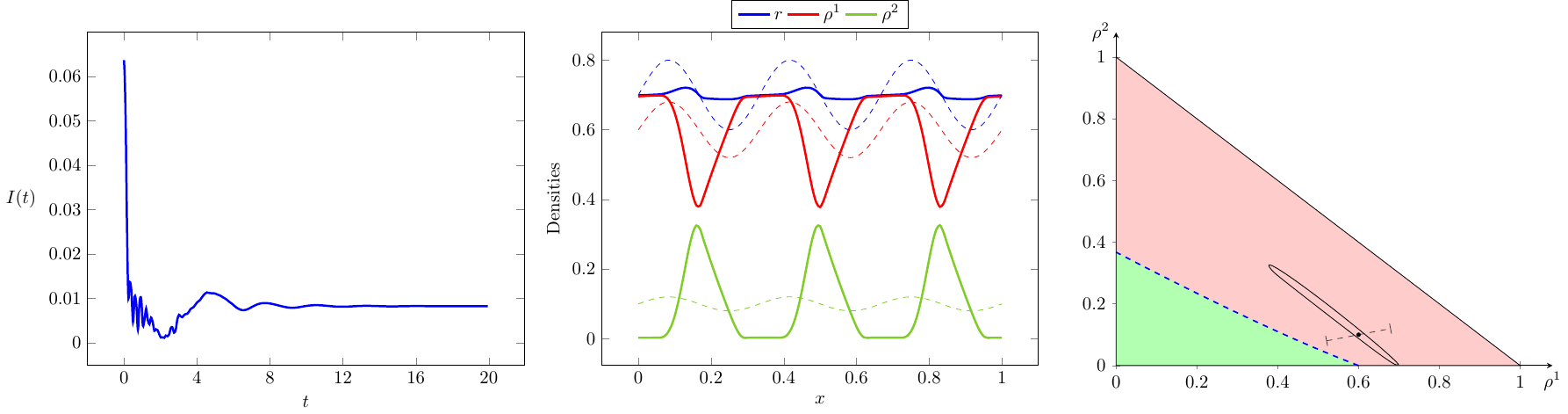}
    \caption{Unstable solution for $\epsilon=1$. Plot of decay of solution $I(t)$ (left), densities at the beginning and end of time interval (center), and $\rho_0(x)$, $\rho(x,t_f=20)$, $x\in[0,1]$, represented within the positive and negative areas of $S_2$ (right). For the first row, $\rho_0(x) = (0.3 + 0.1\sin(6\pi x), 0.4 + 0.15\sin(6\pi x))$, while for the second row, $\rho_0(x) = (0.6 + 0.08\sin(6\pi x), 0.1 + 0.02\sin(6\pi x))$.}
	\label{fig:num_test:periodic:unstable:eps1}
\end{figure}

\begin{figure}[!ht]
	\centering
	\includegraphics[scale=0.48]{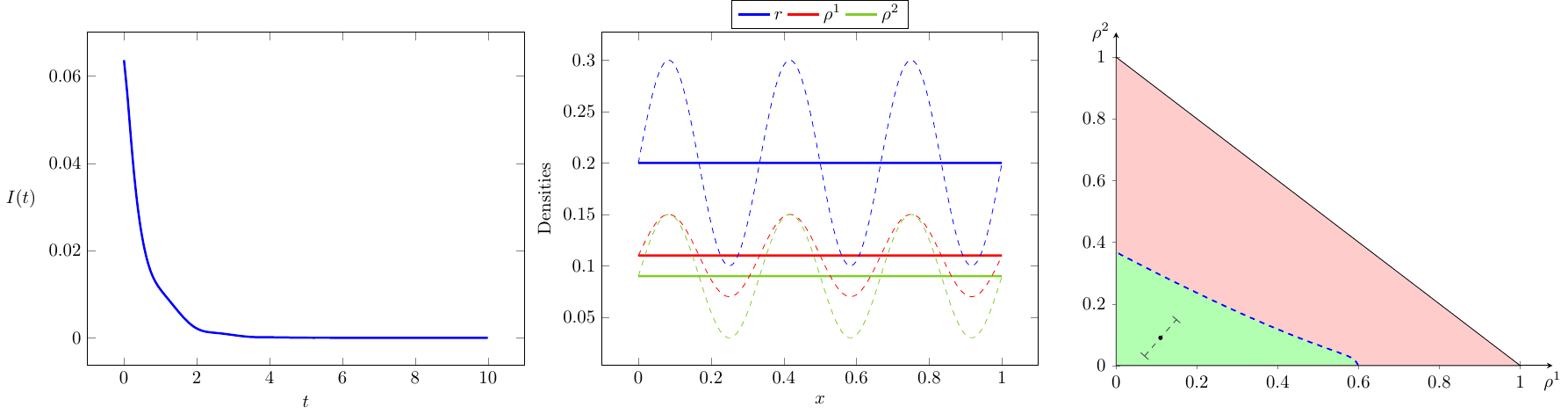}\\
    \includegraphics[scale=0.48]{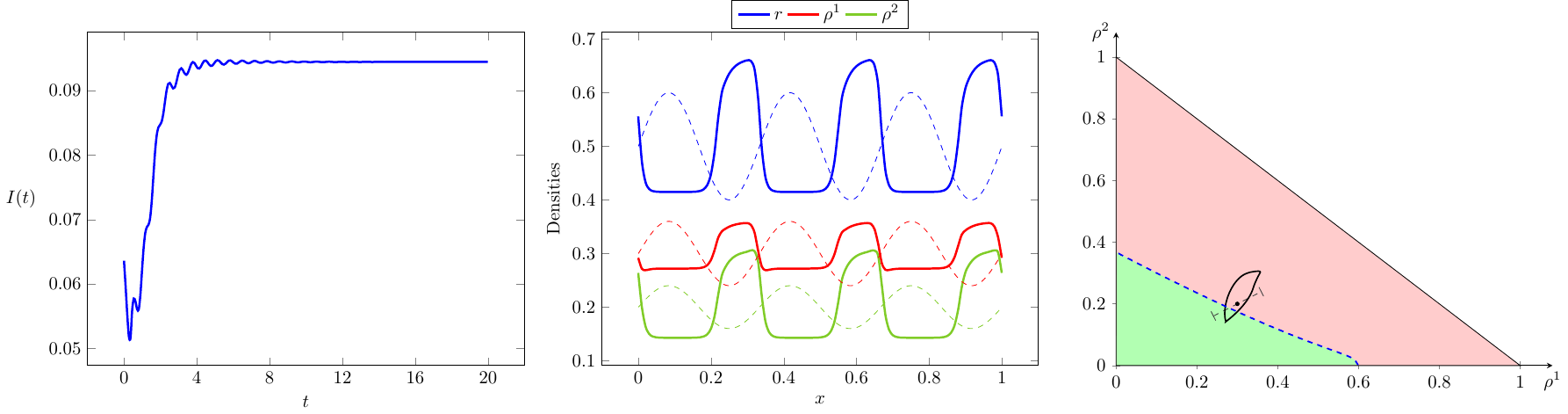}
    \caption{Stable (first row) and unstable (second row) solutions for $\epsilon=0.01$. Plot of decay of solution $I(t)$ (left), densities at the beginning and end of time interval (center), and $\rho_0(x)$, $\rho(x,t_f)$ ($t_f=10$ for the first row, $t_f=20$ for the second), $x\in[0,1]$, represented within the positive and negative areas of $S_2$ (right). For the first row, $\rho_0(x) = (0.11 + 0.09\sin(6\pi x), 0.09 + 0.06\sin(6\pi x))$, while for the second row, $\rho_0(x) = (0.3 + 0.06\sin(6\pi x), 0.2 + 0.04\sin(6\pi x))$.}
	\label{fig:num_test:periodic:stable_unstable:eps001}
\end{figure}

To validate the theoretical stability boundaries derived in Section \ref{sec_kin:stability_cont_model}, we observe the evolution of the system under both stable and unstable regimes in Figures \ref{fig:num_test:periodic:stable:eps1}, \ref{fig:num_test:periodic:unstable:eps1} and \ref{fig:num_test:periodic:stable_unstable:eps001} for different values of $\epsilon$. The plots on the left of each graph represent the decay of solution $I(t)$. The central plots show the class-specific and total densities at the beginning and at the end of the time interval $[t_0=0,t_f]$, represented by dashed and continuous lines, respectively. The graphs on the right show the densities $\rho(x,t_f)=(\rho^1,\rho^2)(x,t_f)$ for which the stability condition $S_2=S_2(\epsilon)$ defined in \eqref{eq_kin:stability_criterion} is either positive (green area) or negative (red area). The boundary within the triangular domain where $S_2=0$ is defined by a dashed blue line. The dashed gray lines in the triangular domain show the initial condition $\rho_0(x)=(\rho^1_0,\rho^2_0)(x)$, while the dot represents the average density $\rho_m=(\rho_m^1,\rho_m^2)$. 

In the stable cases (see Figure \ref{fig:num_test:periodic:stable:eps1} and the first row of Figure \ref{fig:num_test:periodic:stable_unstable:eps001}), i.e., where the average $\rho_m$ lies within the region where $S_2>0$, for $t_f=10$, there is a steady decay in $I(t)$ and a dampening of the traffic wave at the end of the time interval, resulting in a collapse of the solution to $\rho_m$. Conversely, in the unstable scenario (Figure \ref{fig:num_test:periodic:unstable:eps1} and second row of Figure \ref{fig:num_test:periodic:stable_unstable:eps001}), which corresponds to those initial conditions that are contained in the region where $S_2<0$, for $t_f=20$, there is not a decay of solution, but rather a persistent amplification of the initial perturbation. Instead of converging to the homogeneous mean state $\rho_m$, the instability shows the formation of stop-and-go traffic dynamics.

\section{Conclusions and outlook}
\label{sec_kin:conclusions}

In this paper, we have proposed and analyzed a multi-class extension of the kinetic traffic flow model in \cite{BK24discrete}, derived from a non-local Prigogine-Herman framework. 
By introducing a discrete-velocity formulation, we effectively captured the heterogeneous dynamics of traffic, specifically the complex interactions, braking behaviors, and overtaking maneuvers between different vehicle classes.

We proved the hyperbolicity and total linear degeneracy of the system and, where possible, identified the corresponding conservative variables.
Furthermore, to overcome the analytical and numerical challenges posed by non-conservative products arising from the interactions between classes of vehicles, we successfully implemented a path-conservative finite volume scheme. 
Our numerical tests validate the robustness and accuracy of this numerical approach in resolving the discontinuities inherent in multi-class vehicular dynamics. 
Additionally, we established the formal connection between the kinetic framework and its macroscopic limit. 
The numerical simulations also successfully reproduce realistic macroscopic behaviors, such as faster vehicle classes overtaking slower ones, and numerically confirm the theoretical stability boundaries derived through the Chapman-Enskog expansion.

Future research will focus on extending this multi-class discrete-velocity model to complex transportation networks, which will require the formulation of appropriate coupling conditions and numerical fluxes at junctions, and the study of  macroscopic approximations  of the model, similar to \cite{BK24macro}. 

\section*{Acknowledgments}

This work was funded by the European Union's Horizon Europe research and innovation programme under the Marie Sklodowska-Curie Doctoral Network Datahyking (Grant No. 101072546).

\section*{Data availability statement}

The source code used to generate these results is openly available on GitLab at \url{https://gitlab.com/carmen.mezquitanieto/multi-class-kinetic-traffic-flow-model/}.

\begingroup
	\small
	\bibliographystyle{siam}
	\bibliography{bibarticle_kinetic_multiclass}
\endgroup

\end{document}